\documentclass[reqno,11pt]{amsart}
\usepackage{amssymb, amsmath,latexsym,amsfonts,amsbsy, amsthm}
\usepackage{color}
\usepackage{mathrsfs}
\usepackage{esint}
\usepackage{graphics,color}
\usepackage{enumerate}
\usepackage{marginnote}

\newcommand{\beq}{\begin{equation}}
\newcommand{\eeq}{\end{equation}}
\newcommand{\ben}{\begin{eqnarray}}
\newcommand{\een}{\end{eqnarray}}
\newcommand{\beno}{\begin{eqnarray*}}
\newcommand{\eeno}{\end{eqnarray*}}

\renewcommand{\theequation}{\thesection.\arabic{equation}}

\newtheorem{theorem}{Theorem}[section]

\newtheorem{lemma}[theorem]{Lemma}
\newtheorem{proposition}[theorem]{Proposition}

\newtheorem{Theorem}{Theorem}[section]
\newtheorem{Definition}[Theorem]{Definition}

\newtheorem{Lemma}[Theorem]{Lemma}

\newtheorem{Remark}[Theorem]{Remark}

\newcommand{\NN}{\mathbf{N}}

\newcommand{\Id}{\mathrm{Id}}

\begin{document}

\title[Boussinesq Equation]
{ On The Continuous Periodic Weak Solutions Of Boussinesq Equations}

\author{Tao Tao}
\address{School of  Mathematics Sciences, Peking University, Beijing, China}
\email{taotao@amss.ac.cn}
\author{Liqun Zhang }
\address{Academy of mathematic and system science , Beijing , China}
\email{lqzhang@math.ac.cn}
%
%

\date{\today}
\maketitle

\renewcommand{\theequation}{\thesection.\arabic{equation}}
\setcounter{equation}{0}

\begin{abstract}
The Boussingesq equations was introduced in understanding the coupling nature of the thermodynamics and the fluid dynamics.
We show the existence of continuous periodic weak solutions of the Boussinesq equations which satisfies the prescribed kinetic energy or some other prescribed property. Our results represent the conversions between internal energy and mechanical energy.
\end{abstract}


\qquad\qquad
  energy\

\vskip 0.2cm


\section{Introduction}
In this paper, we consider the following Boussinesq equations
\begin{equation}\label{e:boussinesq equation}
\begin{cases}
v_{t}+\hbox{div}(v\otimes v)+\nabla p=\theta e_{3},\quad\quad \mbox{in}\quad T^3\times [0,1]\\
\hbox{div}v=0,\\
\theta_{t}+\hbox{div}(v\theta)=h, \quad\quad \mbox{in}\quad T^3\times [0,1],
\end{cases}
\end{equation}
where $T^3$ denote the 3-dimensional torus, i.e. $T^3=S^1\times S^1 \times S^1$ and $e_3=(0,0,1)^T$. Here in our notations, $v$ is the velocity vector, $p$ is the pressure, $\theta$ denotes the temperature or density which is a scalar function and $h$ is the heat sources.
The Boussinesq equations model many geophysical flows, such as atmospheric fronts and ocean circulations (see, for example, \cite{Ma},\cite{Pe}).

A recurrent issue in the modern theory of PDEs is that one needs to go beyond classical solutions. The pair $(v,\theta)$ on $T^3\times[0,1]$ is called a weak solution of (\ref{e:boussinesq equation}) if they solve (\ref{e:boussinesq equation})  in the following sense:
\begin{align}
\int_0^1\int_{T^3}(\partial_t\varphi\cdot v+\nabla\varphi:v\otimes v+p\hbox{div}\varphi-\theta e_3\cdot\varphi )dxdt=0,\nonumber
\end{align}
for all $\varphi\in C_c^\infty(T^3\times(0,1);R^3),$
\begin{align}
\int_0^1\int_{T^3}(\partial_t\phi\theta+v\cdot\nabla\phi\theta+h\phi)dxdt=0,\nonumber
\end{align}
for all $\phi\in C_c^\infty(T^3\times(0,1);R)$ and
\begin{align}
\int_0^1\int_{T^3}v\cdot\nabla\psi dxdt=0.\nonumber
\end{align}
for all $\psi\in C_c^\infty(T^3\times(0,1);R).$

Moreover, the study of weak solutions in fluid dynamics, including those which fail to conserve energy, is natural in the context of turbulent flows. One of the famous example is the Onsager conjecture on Euler equation which claims that the incompressible Euler equation admits H\"{o}lder continuous solution which dissipates kinetic energy. More precisely, the Onsager conjecture on Euler equation can be stated as following:
\begin{enumerate}
  \item $C^{0,\alpha}$ solution are energy conservative when $\alpha>\frac{1}{3}$.\\
  \item For any $\alpha <\frac{1}{3}$, there exist dissipative solutions with $C^{0,\alpha}$ regularity .
\end{enumerate}

For this conjecture, the part (a) has been proved by P. Constantin, E, Weinan and  E. Titi in \cite{CET}. Slightly weak assumption on the solution was subsequently shown to be sufficient for energy conservation by P. Constantin, etc. in \cite{CPFR,DUR}, also see \cite{SH}. More recently, P. Isett and Sung-jin Oh give a proof to this part of Onsager's conjecture for the Euler equations on manifolds by heat flow method in \cite{ISOH1}.

 The part (b) has been treated by many authors. For weak solutions, there are some results by V. Scheffer (\cite{VS}), A. Shnirelman (\cite{ASH1,ASH2}) and Camillo De Lellis, L\'{a}szl\'{o} Sz\'{e}kelyhidi (\cite{CDL,CDL0}). In recent year, a great progress in the construction of continuous and H\"{o}lder solution was made by Camillo De Lellis, L\'{a}szl\'{o} Sz\'{e}kelyhidi, P. Isett and T. Buckmaster. Camillo De Lellis and L\'{a}szl\'{o} Sz\'{e}kelyhidi developed an iterative scheme (some kind of convex integration) in \cite{CDL2}. Together with the aid of Beltrami flow on $T^3$ and Geometric Lemma, they constructed a continuous periodic solution on $T^3$ which satisfies the prescribed kinetic energy. The solution is a superposition of infinitely many (perturbed) and weakly interacting Beltrami flows.  Building on the iterative techniques in \cite{CDL2} and Nash-Moser mollify techniques, they constructed H\"{o}lder continuous periodic solution with exponent $\theta$, for any $\theta<\frac{1}{10}$, which satisfies the prescribed kinetic energy in \cite{CDL3}. Later, in his PhD Thesis \cite{IS1}, P. Isett introduced some ideas and  construct H\"{o}lder continuous periodic solutions with  any $\theta<\frac{1}{5}$, and the solution has compact supports in time. Adhering more to the original scheme and introducing some new devices in \cite{BCDL1}, Camillo De Lellis, L\'{a}szl\'{o} Sz\'{e}kelyhidi and T. Buckmaster constructed H\"{o}lder continuous weak solutions with any $\theta<\frac{1}{5}$, which satisfies the prescribed kinetic energy, also see \cite{BCDLI}.
Also, in whole space $R^3$, P. Isett and Sung-jin Oh in \cite{ISOH2} constructed H\"{o}lder continuous solutions  with  any $\theta<\frac{1}{5}$, which satisfies the prescribed kinetic energy or is a perturbation of smooth Euler flow.

For the Onsager critical spatial regularity (H\"{o}lder exponent $\theta=\frac{1}{3}$), there are also some interesting results. In line with previous scheme,  by keeping track of sharper, time localized estimates in \cite{TBU}, T. Buckmaster constructed H\"{o}lder continuous(with exponent $\theta<\frac{1}{5}-\varepsilon$ in time-space) periodic solutions which for almost every time belongs to $C_x^{\theta}$, for any $\theta<\frac{1}{3}$, and is compactly temporal supported. By the subdivision of the time interval and careful smoothing Reynolds stress for some time intervals in \cite{BCDL2} , Camillo De Lellis, L\'{a}szl\'{o} Sz\'{e}kelyhidi and T. Buckmaster constructed  H\"{o}lder continuous periodic solution which belongs to $L^1_tC_x^\theta$, for any $\theta<\frac{1}{3}$, and has compact support in time. Recently, S. Daneri consider the Cauchy problem for dissipative H\"{o}lder Euler flow in \cite{DA}.

Motivated by Onsager's conjecture of Euler equation and the above earlier works, we consider the Boussinesq equations and want to know if the similar phenomena can also happen when considering the temperature effects. The difference is that there are conversions between internal energy and mechanical energy. What we need is to overcome the difficulty of interactions between velocity and temperature. To this end, we first consider the existence of continuous periodic solution of Boussinesq equations which satisfies the prescribed kinetic energy. We follow the general scheme in the construction of Euler equations. By establishing the corresponding geometric lemma and improving the iteration scheme, we obtained some existence results. Besides, our results represent some conversions between internal energy and mechanical energy.

Before presenting our theorem, we introduce some notations
\begin{align}
\Theta:=\{\theta(x)\in C^{\infty}(T^{3}): \theta~~ \mbox{only dependent on}~~ x_{3}, \quad i.e. \quad \theta(x)=\overline{\theta}(x_{3})\},\nonumber
\end{align}
and
\begin{align}
\Xi:=\{a(t)b(x_3): a\in C^{\infty}([0,1])~~ \mbox{and} ~~b \in C^{\infty}([0,2\pi])\}.\nonumber
\end{align}
First, for the case without heat source,
i.e. $h=0$, we have the following theorem:
\begin{theorem}\label{t:main 1}
Assume that $e(t):[0,1]\rightarrow R$ is a given positive smooth function and $\theta_0\in\Theta$. Then there exist
\begin{align}
    (v,p,\theta)\in C(T^{3}\times[0,1];R^{3}\times R\times R)\nonumber
\end{align}
and a positive number $M=M(e)$ such that they
solve the system \eqref{e:boussinesq equation} with $h=0$ in the sense of distribution and
$$ e(t)=\int_{T^{3}}|v|^{2}(x,t)dx,\qquad \|\theta-\theta_0\|_0< 4M ,$$
\end{theorem}
where $\|f\|_0=\hbox{sup}_{x,t}|f(x,t)|.$

\begin{Remark}
In our theorem \ref{t:main 1}, if $\theta=0$, then it's the continuous Euler flow with prescribed kinetic energy and have been constructed by Camillo De Lellis and L\'{a}szl\'{o} Sz\'{e}kelyhidi in \cite{CDL2}. In general, for example, if we take $\theta_0=10Mcosx_3\in\Theta$, then we must have $\theta\neq 0$.
\end{Remark}


Next, we consider the Boussinesq equations \eqref{e:boussinesq equation} with heat source. For some kinds of heat sources, we can also construct continuous periodic solution with prescribed kinetic energy. Specially, we have the following theorem:
\begin{theorem}\label{t:main 2}
Assume that $e(t):[0,1]\rightarrow R$ is a positive smooth function and $h\in\Xi$. Then there exist
\begin{align}
    (v,p,\theta)\in C(T^{3}\times[0,1];R^{3}\times R\times R)\nonumber
\end{align}
such that they
solve the system \eqref{e:boussinesq equation} in the sense of distribution and
$$ e(t)=\int_{T^{3}}|v|^{2}(x,t)dx.$$
\end{theorem}

Furthermore, to explore the possible effect of temperature on the velocity field, we can also prove the following theorem for Boussinesq equations without heat source.
\begin{theorem}\label{t:main 3}
For a given positive constant $M$ and any positive number $\lambda$, there exist
\begin{align}
    (v,p,\theta)\in C(T^{3}\times[0,1];R^{3}\times R\times R)\nonumber
\end{align}
such that they
solve the system \eqref{e:boussinesq equation} in the sense of distribution and
\begin{align}
\|v(x,0)\|_0\leq& 4M, ~~~~~ \int_0^1\int_{T^3}|\theta|^2(x,t)dxdt\geq\lambda^2,\nonumber\\
{\sup}_{x \in T^3}|v(x,t)|\geq&\lambda,\qquad
{\inf}_{x \in T^3}|v(x,t)|\leq 4M, \qquad \forall t\in[\frac{1}{2},1].\nonumber
\end{align}
\end{theorem}

\begin{Remark}
For the Boussinesq system on $T^3$, even the initial velocity is small, the oscillation of velocity after sometime could be as large as possible if we have enough thermo in the systems.
\end{Remark}

The result of last theorem coincide with some nature observations and may be helpful in understanding some nature phenomenon.

\begin{Remark}
The above theorems also hold for the two-dimensional Boussinesq system on $T^2$.
\end{Remark}

Finally, we brief some main ideas in our proof. In \cite{CDL2}, they make use of a special families of stationary, plane wave solutions of Euler equation--Beltrami flow--which allows for the control of interference terms between different waves in the construction. For the Boussinesq equations, we don't know the existence of the analogous families of stationary, plane wave solutions. Following an idea in \cite{NASH,ISV2}, we make use of a multi-steps iteration scheme and one-dimensional oscillation. That is, in each step, we add several plane waves which oscillate along the same direction with different frequency. Thus we can remove one component of stress error in each step. To reduce the whole stress errors, we then divide into several steps and finish it one by one. Moreover, in Boussinesq equations, the velocity and temperature are coupled together, so we need to reduce two stress errors simultaneously. To achieve this, we first extend the geometric lemma proved in \cite{CDL2} and add an associated plane wave in the velocity and temperature simultaneously in each step. The consistence is important for us to reduce the temperature stress errors in our construction of continuous temperature.

\section{Proof of main result and Plan of the paper}

As \cite{CDL2}, the proof of Theorem \ref{t:main 1}, Theorem \ref{t:main 2} and Theorem \ref{t:main 3} will be achieved through an iteration procedure. In this and subsequent section, $\mathcal{S}_0^{3\times3}$  denotes the vector space of
symmetric trace-free $3\times3$ matrices.
\begin{Definition}\label{d:boussinesq reynold}
Assume $v,p,\theta,\mathring{R},f$ are smooth functions on $T^{3}\times[0,1]$ taking values,~respectively,~in $R^{3},R,R,\mathcal{S}_{0}^{3\times3},R^{3}$.
We say that they solve the Boussinesq-Reynolds system (with or without heat source) if
\begin{equation}\label{d:boussinesq reynold}
\begin{cases}
\partial_tv+{\rm div}(v\otimes v)+\nabla p=\theta e_{3}+{\rm div}\mathring{R}\\
{\rm div}v=0\\
\theta_{t}+{\rm div}(v\theta)=h+{\rm div}f.
\end{cases}
\end{equation}
\end{Definition}
We now state the first main proposition of the paper,~of which Theorem \ref{t:main 1} and Theorem  \ref{t:main 2}  are implied directly. The proposition is also an extension of the corresponding result given in \cite{CDL2}.
\begin{proposition}\label{p: iterative 1}
Let $e(t)$ be as in Theorem \ref{t:main 1} and Theorem \ref{t:main 2}. ~Then we can choose two positive constants $\eta$ and $M$ only dependent of $e(t)$,~such that the following properties hold:

For any $0< \delta\leq1$,~if $(v_,~ p, ~\theta, ~\mathring{R}, ~f)\in C^{\infty}([0,1]\times T^{3})$ solve Boussinesq-Reynolds system (\ref{d:boussinesq reynold}) and
\begin{align}
    \frac{3\delta}{4}e(t)\leq e(t)-\int_{T^{3}}|v|^{2}(x,t)dx\leq&\frac{5\delta}{4}e(t),\quad \forall t \in [0,1], \label{e:erengy initial}\\
    \sup_{x,t}|\mathring{R}(x,t)|\leq& \eta\delta,\label{e:reynold initial}\\ \sup_{x,t}|f(x,t)|\leq& \eta\delta,\label{e:reynold initial 2}
\end{align}
then we can construct new functions $(\tilde{v},~\tilde{p},~\tilde{\theta},~\mathring{\tilde{R}},~\tilde{f})\in C^{\infty}([0,1]\times T^{3})$ , they also solve Boussinesq-Reynolds system (\ref{d:boussinesq reynold}) and satisfy
\begin{align}
    \frac{3\delta}{8}e(t)\leq e(t)-\int_{T^{3}}|\tilde{v}|^{2}(x,t)dx\leq&\frac{5\delta}{8}e(t),\quad \forall t \in [0,1],\\
     \sup_{x,t}|\mathring{\tilde{R}}(x,t)|\leq&\frac{\eta\delta}{2} ,\\
     \sup_{x,t}|\tilde{f}(x,t)|\leq& \frac{\eta\delta}{2},\\
    \sup_{x,t}|\tilde{v}(x,t)-v(x,t)| \leq& M\sqrt{\delta},\\
    \sup_{x,t}|\tilde{\theta}(x,t)-\theta(x,t)| \leq& M\sqrt{\delta},\\
    \sup_{x,t}|\tilde{p}(x,t)-p(x,t)| \leq& M\sqrt{\delta}.
    \end{align}
\end{proposition}
We will prove Proposition \ref{p: iterative 1} in the subsequent sections.~First,~we give a proof of Theorem \ref{t:main 1} and Theorem \ref{t:main 2} by using this proposition.
\begin{proof}[Proof of Theorem 1.1]
In this case, $h=0$.
We first set
\begin{align}
v_{0}:=0,~ \theta_0\in \Theta, ~p_{0}:=\int_{0}^{x_3}\theta_0(y)dy,~\mathring{R}_{0}:=0,~f_{0}:=0\nonumber
\end{align}and $\delta=1$.
Obviously, they solve Boussinesq-Reynolds system (\ref{d:boussinesq reynold}) and satisfy the following estimates
\begin{align}
\frac{3\delta}{4}e(t)\leq e(t)-\int_{T^{3}}|v_{0}|^{2}(x,t)dx\leq&\frac{5\delta}{4}e(t),\qquad \forall t \in [0,1]\nonumber\\
    \sup_{x,t}|\mathring{R}_{0}(x,t)|=0 \quad(\leq& \eta\delta),\nonumber\\
     \sup_{x,t}|f_{0}(x,t)|=0 \quad(\leq& \eta\delta).\nonumber
\end{align}
Then  using Proposition \ref{p: iterative 1} iteratively, we can construct a sequence $(v_{n},~p_{n},~\theta_{n},~\mathring{R_{n}},~f_{n})$,
which solve (\ref{d:boussinesq reynold}) and satisfy
\begin{eqnarray}
    \frac{3}{4}\frac{e(t)}{2^{n}}\leq e(t)-\int_{T^{3}}|v_{n}|^{2}(x,t)dx\leq&\frac{5}{4}\frac{e(t)}{2^{n}},\quad \forall t \in [0,1]\label{e:energy_final}\\
     \sup_{x,t}|\mathring{R_{n}}(x,t)|\leq\frac{\eta}{2^{n}} ,\label{e:stress_final}\\\sup_{x,t}|f_{n}(x,t)|\leq \frac{\eta}{2^{n}},\\
    \sup_{x,t}|v_{n+1}(x,t)-v_{n}(x, t)| \leq M\sqrt{\frac{1}{2^{n}}},\label{e:velocity_final}\\
    \sup_{x,t}|\theta_{n+1}(x,t)-\theta_{n}(x,t)| \leq M\sqrt{\frac{1}{2^{n}}},\label{e:temperature_final}\\
    \sup_{x,t}|p_{n+1}(x,t)-p_{n}(x,t)| \leq M\sqrt{\frac{1}{2^{n}}}.\label{e:pressure_final}
\end{eqnarray}
Therefore from (\ref{e:stress_final})-(\ref{e:pressure_final}), we know that $(v_{n},~p_{n},~\theta_{n},~\mathring{R}_{n},~f_{n})$ are Cauchy sequence in $C(T^{3}\times[0,1])$, therefore there exist
\[\begin{aligned}
    (v,p,\theta)\in C(T^{3}\times[0,1])
\end{aligned}\]
such that
\[\begin{aligned}
    v_{n}\rightarrow v,\qquad p_{n}\rightarrow p,\qquad \theta_{n}\rightarrow \theta,
    \qquad \mathring{R_{n}}\rightarrow 0,\qquad f_{n}\rightarrow 0.
\end{aligned}\]
in $C(T^{3}\times[0,1])$ as $n \rightarrow \infty$.

Moreover,~by \eqref{e:energy_final},
$$ e(t)=\int_{T^{3}}|v|^{2}(x,t)dx \qquad \forall t\in[0,1].$$
By (2.16),
$$\|\theta-\theta_0\|_0\leq M\sum\limits_{n=0}^{\infty}\sqrt{\frac{1}{2^{n}}}< 4M.$$
Passing into the limit in (\ref{d:boussinesq reynold}), we conclude that $v,~p,~\theta$ solve (\ref{e:boussinesq equation}) in the sense of distribution.
\end{proof}

\begin{proof}[Proof of Theorem 1.2]
For any $h\in\Xi$, there are two functions $a\in C^\infty$ and $b\in C^\infty$ such that $h(t,x_3)=a(t)b(x_3)$.
Thus if
we set
\begin{align}
v_{0}:=0,~\quad \theta_0:=\int_0^ta(s)ds b(x_3), ~\quad p_{0}:=\int_0^ta(s)ds \int_{0}^{x_3}b(y)dy,~\quad\mathring{R}_{0}:=0,~\quad f_{0}:=0\nonumber
\end{align}and $\delta=1$.
Then they solve (\ref{d:boussinesq reynold}) and satisfy the following estimates
\begin{align}
\frac{3\delta}{4}e(t)\leq e(t)-\int_{T^{3}}|v_{0}|^{2}(x,t)dx\leq&\frac{5\delta}{4}e(t),\qquad \forall t \in [0,1]\nonumber\\
    \sup_{x,t}|\mathring{R}_{0}(x,t)|=0(\leq& \eta\delta),\nonumber\\
     \sup_{x,t}|f_{0}(x,t)|=0(\leq& \eta\delta).\nonumber
\end{align}
Then we can complete our proof by following the proof of Theorem \ref{t:main 1} line by line. We omit the detail here.
\end{proof}

To prove Theorem \ref{t:main 3}, we again make use of iterative techniques. We state the following iterative proposition, which implies Theorem \ref{t:main 3}.
\begin{proposition}\label{p:iterative 2}
There exist two absolute constants $M$ and $\eta$ such that the following properties hold:

For any $0< \delta\leq1$,~if $(v_,~ p, ~\theta, ~\mathring{R}, ~f)\in C^{\infty}([0,1]\times T^{3})$ solve Boussinesq-Reynolds system (\ref{d:boussinesq reynold}) and
\begin{align}
    \sup_{x,t}|\mathring{R}(x,t)|\leq& \eta\delta,\\ \sup_{x,t}|f(x,t)|\leq& \eta\delta,
\end{align}
then we can construct new functions $(\tilde{v},~\tilde{p},~\tilde{\theta},~\mathring{\tilde{R}},~\tilde{f})\in C^{\infty}([0,1]\times T^{3})$, they also solve (\ref{d:boussinesq reynold}) and satisfy
\begin{align}
     \sup_{x,t}|\mathring{\tilde{R}}(x,t)|\leq&\frac{\eta}{2} \delta,\\
     \sup_{x,t}|\tilde{f}(x,t)|\leq& \frac{\eta}{2}\delta,\\
    \sup_{x,t}|\tilde{v}(x,t)-v(x,t)| \leq& M\sqrt{\delta},\\
    \sup_{x,t}|\tilde{\theta}(x,t)-\theta(x,t)| \leq& M\sqrt{\delta},\\
    \sup_{x,t}|\tilde{p}(x,t)-p(x,t)| \leq& M\sqrt{\delta}.
\end{align}
\end{proposition}
\begin{proof}[Proof of Theorem \ref{t:main 3}]
First,
we set
\begin{align}
v_0=&\left(\begin{array}{c}
tN sin(N^2x_2)\\
0\\
0
\end{array}\right),\quad
\mathring{R_0}=
\left(\begin{array}{ccc}
0  & -\frac{cos(N^2x_2)}{N} & 0\\
-\frac{cos(N^2x_2)}{N} & 0 & 0\\
0 &  0  &  0
\end{array}\right),\quad
f_0=\left(\begin{array}{c}
0\\
0\\
\frac{cos(N^2x_3)}{N}
\end{array}\right),\nonumber\\
p_{0}=&-(1-t)\frac{cos(N^2x_3)}{N},
 \qquad \qquad \theta_0=(1-t)Nsin(N^2x_3),\nonumber
\end{align}and $\delta=1$.
Then they solve (\ref{d:boussinesq reynold}). If we take $N\geq \frac{2}{\eta}$, then they satisfy the following estimates
\begin{align}
    \sup_{x,t}|\mathring{R}_{0}(x,t)|\leq& \eta\delta,\nonumber\\
     \sup_{x,t}|f_{0}(x,t)|\leq& \eta\delta.\nonumber
\end{align}
By using Proposition \ref{p:iterative 2} iteratively, we can construct a sequence $(v_{n},~p_{n},~\theta_{n},~\mathring{R_{n}},~f_{n})$,
which solve (\ref{d:boussinesq reynold}) and satisfy
\begin{align}
     \sup_{x,t}|\mathring{R_{n}}(x,t)|\leq&\frac{\eta}{2^{n}} ,\\\sup_{x,t}|f_{n}(x,t)|\leq& \frac{\eta}{2^{n}},\\
    \sup_{x,t}|v_{n+1}(x,t)-v_{n}(x, t)| \leq& M\sqrt{\frac{1}{2^{n}}},\label{i:velocity final}\\
    \sup_{x,t}|\theta_{n+1}(x,t)-\theta_{n}(x,t)| \leq& M\sqrt{\frac{1}{2^{n}}},\label{i:temperture final}\\
    \sup_{x,t}|p_{n+1}(x,t)-p_{n}(x,t)| \leq& M\sqrt{\frac{1}{2^{n}}}.
\end{align}
Then we know that $(v_{n},~p_{n},~\theta_{n},~\mathring{R}_{n},~f_{n})$ are Cauchy sequence in $C(T^{3}\times[0,1])$, therefore there exist
\[\begin{aligned}
    (v,p,\theta)\in C(T^{3}\times[0,1])
\end{aligned}\]
such that
\[\begin{aligned}
    v_{n}\rightarrow v,\qquad p_{n}\rightarrow p,\qquad \theta_{n}\rightarrow \theta,
    \qquad \mathring{R_{n}}\rightarrow 0,\qquad f_{n}\rightarrow 0,
\end{aligned}\]
in $C(T^{3}\times[0,1])$ as $n \rightarrow \infty$.

By (\ref{i:velocity final}) and (\ref{i:temperture final}), we have
$$\|v-v_0\|_0\leq M\sum\limits_{n=0}^{\infty}\sqrt{\frac{1}{2^{n}}}< 4M,$$
and
$$\|\theta-\theta_0\|_0\leq M\sum\limits_{n=0}^{\infty}\sqrt{\frac{1}{2^{n}}}< 4M.$$
Finally, let $\lambda$ be as in Theorem \ref{t:main 3} and take $N={\rm max}\{\frac{2}{\eta}, 4\lambda, 16M\}$, then for $t\in[\frac{1}{2},1]$
\begin{align}
{\sup_{x\in T^3}}|v(x,t)|\geq {\sup}_{x\in T^3}|v_0(x,t)|-4M\geq \frac{N}{4}\geq\lambda,\nonumber\\
{\inf}_{x\in T^3}|v(x,t)|\leq {\inf}_{x\in T^3}|v_0(x,t)|+4M\leq4M.\qquad \nonumber
\end{align}
Moreover, since $v_{0}(x,0)=0$ , we have
\begin{align}
\|v(x,0)\|_0\leq 4M. \nonumber
\end{align}
A direct calculation gives,
\begin{align}
\int_0^1\int_{T^3}|\theta_0|^2(x,t)dxdt=\frac{4\pi^3}{3}N^2,\nonumber
\end{align}
therefore
 \begin{align}
 &\int_0^1\int_{T^3}|\theta|^2(x,t)dxdt\nonumber\\
 \geq&\frac{1}{2}\int_0^1\int_{T^3}|\theta_0|^2(x,t)dxdt-\int_0^1\int_{T^3}|\theta-\theta_0|^2(x,t)dxdt\nonumber\\
 \geq&\frac{2\pi^3}{3}N^2-(2\pi)^3(4M)^2\nonumber\\
 \geq&\lambda^2.\nonumber
 \end{align}
Passing into the limit in (\ref{d:boussinesq reynold}) we conclude that $v,~p,~\theta$ solve (\ref{e:boussinesq equation}) in the sense of distribution.
\end{proof}


\subsection{Outline of the proof of propositions}

\indent

The constructions of the functions $\tilde{v}, \tilde{\theta}$ consist of several steps. In the first steps, we add perturbations to $v_0, \theta_0$  and get new functions $v_{01}, \theta_{01}$ as
 \begin{align}
 v_{01}=&v_0+w_{1o}+w_{1oc}:=v_0+w_1,\nonumber \\
 \theta_{01}=&\theta_0+\chi_1.\nonumber
 \end{align}
where $w_{1o}, w_{1oc}, \chi_1$ are highly oscillated functions and given by explicit formulas. We introduce two parameters $\mu_1,~\lambda_1$ in the construction of perturbation, which satisfy $\mu_1,~\lambda_1,~\frac{\lambda_1}{\mu_1}\in {\rm N}$. After the construction of $v_{01}, \theta_{01}$, we then focus on constructing  functions $R_{01}, p_{01}$ and $f_{01}$ which satisfies the desired estimate and solves the system (\ref{d:boussinesq reynold}). After the first step, the stress becomes smaller in the following sense: if
\begin{align}
\rho(t)Id-\mathring{R_0}=&\sum\limits_{i=1}^{L}a_i^2\Big(Id-\frac{k_i}{|k_i|}\otimes\frac{k_i}{|k_i|}\Big),\nonumber\\
f_0=&\sum\limits_{i=1}^{3}b_iA_{k_i},\nonumber
\end{align}
then
\begin{align}
R_{01}=&\sum\limits_{i=2}^{L}a_i^2\Big(Id-\frac{k_i}{|k_i|}\otimes\frac{k_i}{|k_i|}\Big)+\delta R_{01},\nonumber\\
f_{01}=&\sum\limits_{i=2}^{3}b_iA_{k_i}+\delta f_{01}.\nonumber
\end{align}
where $\delta R_{01}, \delta f_0$ can be arbitrary small through the appropriate choice on $\mu_1$ and $\lambda_1$.

Repeating the above process, finally we can obtain the needed functions $(\tilde{v},~\tilde{p},~\tilde{\theta},~\mathring{\tilde{R}},~\tilde{f}).$\\

The rest of paper is organized as follows. In section 3, we do some preliminaries. We extend the Gemmetric Lemma in \cite{CDL2} and construct a new operator $\mathcal{G}$ which will be be used in order to deal the
stresses arising from iteration scheme. Then we perform the first step in next several sections. In section 4, we not only define the perturbations $w_{1o}, w_{1oc}, \chi_1$ and new stresses $R_{01}, f_{01}$, but also determine the constants $\eta$ and $M$ appeared in the estimates in Proposition \ref{p: iterative 1}. In section 5 and section 6, we calculate the main parts of $R_{01}, f_{01}$ and prove the relevant estimates of the various terms involved in the construction separately. After completing the first step, in section 7, section 8 and section 9 we then construct $v_{0n}, p_{0n}, \theta_{0n}, R_{0n}, f_{0n}$ and prove relevant estimates
by inductions. The construction in section 7 is quite similar to that of the first step given in section 4. We calculate the main error parts $R_{0n}, f_{0n}$ in section 8 and obtain various error estimates in section 9 respectively. Those two sections are also similar to that of section 5 and section 6 respectively. And we obtain the energy estimates in section 10. Finally, in section 11, we give a proof of proposition 2.1 by choosing appropriate parameters $\mu_n,\lambda_n$ for $1\leq n\leq L$. After the proof of proposition \ref{p: iterative 1}, we give a proof of proposition \ref{p:iterative 2} and this is done in section 12.

\section{Preliminaries}

As \cite{CDL2}, we introduction some notations: $R^{3\times 3}$ denotes the space of $3\times 3$ matrices; $\mathcal{S}^{3\times 3}$
 denotes the spaces of $3\times 3$ symmetric matrices
and $\Id$ denotes $3\times 3$ identity matrix .
The matrix norm $|R|:=\max_{1\leq i,j\leq 3}|R_{ij}|$ if $R=(R_{ij})_{3\times 3}$.
\subsection{Geometric Lemma}

\indent

The following lemma is an extension of geometric lemma given in \cite{CDL2} to our case. Within it, we not only represent a prescribed symmetric matrix $R$, but also a prescribed vector simultaneously.

\begin{Lemma}[Geometric Lemma]\label{p:split}
For every $N\in {\NN}$, we can choose $r_{0}>0$ and $\bar{\lambda}>1$ such that the following property holds: There exist pairwise disjoint subsets
\begin{align*}
    \Lambda_{j}\subseteq \{k\in Z^{3}:|k|=\bar{\lambda}\},\quad j\in\{1,\cdots ,N\},
\end{align*}
smooth positive functions
 \begin{align*}
    \gamma_{k}^{(j)}\in C^{\infty}(B_{r_{0}}(Id)),    j\in\{1,\cdots ,N\},~~k\in\Lambda_{j},
\end{align*}
vectors
 \begin{align*}
   A^j_k\in R^3, ~~|A^j_k|=\frac{1}{\sqrt{2}},~~ k\cdot A^j_k=0,~~j\in\{1,\cdots ,N\},~~k\in\Lambda_{j}
 \end{align*}
and smooth functions
 \begin{align*}
    g_{k}^{(j)}\in C^{\infty}(R^{3}), j\in\{1,\cdots,N\},~~k\in\Lambda_{j},
 \end{align*}
such that
\begin{enumerate}
 \item $k\in\Lambda_{j}$ implies $-k\in\Lambda_{j}$ and $\gamma_{k}^{(j)}=\gamma_{-k}^{(j)};$
  \item for every $R \in B_{r_{0}}(Id)$,~ the following identity holds:$$R=\frac{1}{2}\sum\limits_{k\in \Lambda_{j}}(\gamma_{k}^{(j)}(R))^2\Big(Id-\frac{k}{|k|}\otimes\frac{k}{|k|}\Big);$$
   \item for every $f\in C^{\infty}(R^{3})$,~we have the identity
   $$f=\sum_{k\in\Lambda_{j}}g_{k}^{(j)}(f)A^j_{k}.$$
\end{enumerate}

\end{Lemma}
\begin{proof}
The proof is essentially an extension of it given in \cite{CDL2}. The choosing of $r_0, ~\lambda_0,~\Lambda_j,~\gamma_k^j$ and the results (1), (2) are due to Camillo De Lellis and L\'{a}szl\'{o} Sz\'{e}kelyhidi and can be found in \cite{CDL2}. We only give the proof on existence of $A^j_{k}$, $g^j_{k}$ and the related property.
 For the given $j\in \{1,\cdot\cdot\cdot,N\}$, we can choose $k^j_1,~ k^j_2 ,~k^j_3$ in $\Lambda_{j}$  such that they are linearly independent, and then we can choose 3-dimensional vectors $A^j_{k_1},A^j_{k_2},A^j_{k_3}$ such that  $A^j_{k_i}\cdot k^j_i=0$ and $|A^j_{k_i}|=\frac{1}{\sqrt{2}}$, $i=1,2,3$, which are linear independent.

 For $k\in \Lambda_j\backslash \{k^j_1, k^j_2 ,k^j_3\}$, we take $A^j_k$ be a 3-dimensional vector such that $A^j_k\cdot k=0, |A^j_k|=\frac{1}{\sqrt{2}}$.
 Then, we can choose $g_k^j(f)=0$ for $k\in\Lambda_j\backslash \{k^j_1, k^j_2 ,k^j_3\} $ and
  \[
  \left(\begin{array}{cc}
  g_{k_1}^j(f)\\[12pt]
  g_{k_2}^j(f)\\[12pt]
  g_{k_2}^j(f)
  \end{array}\right)
  :=(A^j_{k_1},A^j_{k_2},A^j_{k_3})^{-1}f.\]
where $f\in C^{\infty}(T^3)$ and $(A^j_{k_1},A^j_{k_2},A^j_{k_3})^{-1}$ denote the inverse matrix of $(A^j_{k_1},A^j_{k_2},A^j_{k_3}).$
Thus, the linear functionals $g_{k}^{(j)}$ satisfies
 \begin{align*}
    \forall f\in C^{\infty}(R^{3}),f=\sum_{k\in \Lambda_{j}}g_{k}^{(j)}(f)A^j_{k}.
  \end{align*}
Then we finished the proof of the geometric lemma.

\end{proof}

\subsection{The operator $\mathcal{R}$ and $\mathcal{G}$}

\indent

We introduce some operators
in order to deal with the Reynolds stresses. The operator $\mathcal{R}$ was introduced in  \cite{CDL2} and the operator $\mathcal{G}$ is given by us.
\subsubsection{The operator $\mathcal{R}$ }
The following lemma is taken from \cite{CDL2}, we copy it here for the completeness of the paper (the proof refers to \cite{CDL2}).
\begin{lemma}[$\mathcal{R}=\textrm{div}^{-1}$]\label{l:reyn}
There exist a linear operator $\mathcal{R}$ from $C^\infty (T^3, R^3)$ to $C^\infty (T^3, R^{3\times3})$ such that the following property hold: for any $v\in C^\infty (T^3, R^3)$ we have
\begin{itemize}\label{l:R}
\item[(a)] $\mathcal{R}v(x)$ is a symmetric trace-free matrix for each $x\in T^3$;
\item[(b)] ${\rm div}\mathcal{R} v = v-\fint_{T^3}v$.
\end{itemize}
\end{lemma}
Here and subsequent, we use the natation $\fint_{T^3}v(x)dx=\frac{1}{|T^3|}\int_{T^3}v(x)dx$.

\subsubsection{The operator $\mathcal{G}$ }

\indent

\begin{Definition}
 Let $b\in C^{\infty}(T^{3};R)$ be a smooth function. We define a vector-valued periodic function $ \mathcal{G}b$ by
\begin{align*}
    \mathcal{G}b=\nabla a
\end{align*}
where $a\in C^{\infty}(T^{3};R)$ is the solution of
\begin{align*}
    \Delta a=b-\fint_{T^{3}}b \textrm{ in } T^3
\end{align*}
 with $\int_{T^{3}}a(x)dx=0$.
 \end{Definition}
\begin{Lemma}\label{l:G inverse}
$(\mathcal{G}={\rm div}^{-1})$For any $b\in C^{\infty}(T^{3};R)$, we have
\begin{align*}
   {\rm div}(\mathcal{G}b)=b-\fint_{T^{3}}b.
\end{align*}

\begin{proof}
The proof is elementary. In fact, we calculate directly
\begin{align*}
   {\rm div}(\mathcal{G}b)=\triangle a=b-\fint_{T^{3}}b.
\end{align*}
\end{proof}
\end{Lemma}
Some estimates related to the operators $\mathcal{R}$  and $\mathcal{G}$ can be founded in the Appendix A.

\section{the constructions of $ v_{01}, p_{01}, \theta_{01}, R_{01}, f_{01} $}
The constructions of $\tilde{v},~\tilde{p},~\tilde{\theta},~\mathring{\tilde{R}},~\tilde{f}$ from  $v,~p,~\theta,~\mathring{R},~f$ consist of several steps. The main idea is to decompose the stress error into blocks with the help of geometric lemma and to remove one block in one step. In this section, we perform the first step.

For convenience, we set $v_0:=v,~p_0:=p,~\theta_0:=\theta,~\mathring{R}_0:=\mathring{R},~f_0:=f$.

\subsection{The main perturbation $w_{1o}$}

\indent

We first introduce some notations. From \cite{CDL2}, we have the following partition of unity: for two constants $c_{1}$ and $c_{2}$ such that $\frac{\sqrt{3}}{2}<c_{1}<c_{2}<1$ , we have a family of functions $\alpha_l\in C_c^{\infty}(R^3):l\in Z^3$ such that
\begin{align}\label{p:unity}
\sum\limits_{l\in Z^3}\alpha_l^2=1,\qquad \hbox{supp}\alpha_l\subseteq B_{c_2}(l).
\end{align}
We apply the geometric Lemma \ref{p:split} with $N=1$ and obtain $\bar{\lambda}>1,r_{0}>0$ , subset $\Lambda=\{\pm k_1,...,\pm k_L\}$ and
 vectors $\{A_{\pm k_j}, j=1,\cdot\cdot\cdot,L\}$ together with corresponding functions
\begin{align*}
    \gamma_{k_i}\in C^{(\infty)}(B_{r_{0}}(Id)),\qquad g_{k_i}\in C^{(\infty)}(R^{3}),\qquad  i=1,\cdots,L.
\end{align*}
where $L$ is a fixed integer.
Thus the result can be restated as following:\\
For any $R\in B_{r_0}(Id)$, we have the identity
\begin{align}\label{p:split 1}
R=\sum\limits_{i=1}^{L}\gamma^2_{k_i}(R)\Big(Id-\frac{k_i}{|k_i|}\otimes\frac{k_i}{|k_i|}\Big).
\end{align}
and for any $f\in C^{(\infty)}(R^{3})$, we have
\begin{align}
f=\sum\limits_{i=1}^{L}g_{k_i}(f)A_{k_i}.\nonumber
\end{align}
In fact, from the construction of $g_{k_i}$, after possibly relabel the index, we have
\begin{align}\label{p:split 2}
f=\sum\limits_{i=1}^{3}g_{k_i}(f)A_{k_i}.
\end{align}

Next, as in \cite{CDL2}, to define the amplitude of the perturbation, we set
\begin{align}\label{d:amp 2}
    \bar{\rho}(t):=\frac{1}{(2\pi)^{3}}\Big{(}e(t)\Big(1-\frac{\delta}{2}\Big)-\int_{T^{3}}|v_0|^{2}(x,t)dx\Big{)},
    \end{align}
    and
    \begin{align}\label{d:matr}
        R_0(x,t):=\bar{\rho}(t)Id-\mathring{R}_0(x,t).
    \end{align}
then for any $l\in Z^3$, we denote $b_{1l}$ by
   \begin{align}\label{d:l amp}
    b_{1l}(x,t):=&\sqrt{\bar{\rho}(t)}\alpha_l(\mu_1 v_0)\gamma_{k_1}\Big{(}\frac{R_0(x,t)}{\bar{\rho}(t)}\Big{)},
    \end{align}
    and define
    \begin{align}
    B_{k_1}:=A_{k_1}+i\frac{k_1}{|k_1|}\times A_{k_{1}}.
   \end{align}
   Then we define $l$-perturbation
   \begin{align}\label{d:w 1ol}
    w_{1ol}:=b_{1l}(x,t)\Big{(}B_{k_1}e^{i\lambda_1 2^{|l|} k_1\cdot (x-\frac{l}{\mu_1}t)}+B_{-k_1}e^{-i\lambda_1 2^{|l|} k_1\cdot (x-\frac{l}{\mu_1}t)}\Big{)}.
   \end{align}
   where we set $A_{-k_1}=A_{k_1}$.

Finally, we define 1-th perturbation
   \begin{align}\label{d:1o}
    w_{1o}:=\sum_{l\in Z^3}w_{1ol}.
   \end{align}

Obviously, $w_{1ol},w_{1o}$ are all real 3-dimensional vector functions.
   According to (\ref{p:unity}), we have $\hbox{supp} \alpha_l\cap \hbox{supp}\alpha_{l'}=\emptyset$  if $|l-l'|\geq2$, therefore the above summation is meaningful.

 \subsection{The constants $\eta$ and $M$}

 \indent

   From \cite{CDL2}, we known that there exist constant $\eta=\eta(e)$ (independent on $\delta$) such that
    \begin{align}\label{b:bound R 0}
    \Big{\|}\frac{R_0}{\bar{\rho}(t)}-Id\Big\|_0\leq \frac{\|\mathring{R}_0\|_0}{c\delta m}\leq \frac{\eta}{cm}\leq\frac{r_0}{2},
   \end{align}
   so $b_{1l}$ is well-defined,
   where $\|f\|_0=\sup_{x,t}|f(x,t)|.$
   Also from \cite{CDL2}, there exists a constant $ M=M(e)>1$ (in particular independent of $\delta$) such that
   \begin{align}
    \|b_{1l}\|_{0}\leq\frac{M\sqrt{\delta}}{10L},\nonumber
   \end{align}
   where $L$ is fixed constant from geometric lemma.\\
   Moreover, from the property (\ref{p:unity}) of $\alpha_l$, we know that
     \begin{align}\label{e:bound 1o}
    \|w_{1o}\|_{0}\leq\frac{M\sqrt{\delta}}{2L}.
   \end{align}

 \subsection{The correction $w_{1oc}$ and the perturbations $w_{1}$ and $\chi_1$}

   \indent

   We denote the $l$-correction
   \begin{align}\label{d:w 1ocl}
   w_{1ocl}:=&\frac{1}{\lambda_1\lambda_{0}}\Big{(}\frac{\nabla b_{1l}(x,t)\times B_{k_1}}{2^{|l|}}e^{i\lambda_1 2^{|l|} k_1\cdot (x-\frac{l}{\mu_1}t)}+\frac{\nabla b_{1l}(x,t)\times B_{-{k_1}}}{2^{|l|}}e^{-i\lambda_1 2^{|l|} k_1\cdot (x-\frac{l}{\mu_1}t)}\Big{)},
   \end{align}
   then denote 1-th correction
   \begin{align}\label{d:w 1oc}
    w_{1oc}:=\sum_{l\in Z^3}w_{1ocl}.
    \end{align}

   Finally, we denote 1-th perturbation
   \begin{align}\label{d:w 1}
    w_1:=w_{1o}+w_{1oc}.
   \end{align}
  Thus, if we denote $w_{1l}$ by
   \begin{align}\label{d:w 1l}
   w_{1l}:=&w_{1ol}+w_{1ocl}\nonumber\\
   =&\frac{1}{\lambda_1\lambda_{0}}\hbox{curl}\Big(\frac{b_{1l}(x,t)B_{k_1}}{2^{|l|}}e^{i\lambda_1 2^{|l|} k_1\cdot (x-\frac{l}{\mu_1}t)}+\frac{b_{1l}(x,t)B_{-{k_1}}}{2^{|l|}}e^{-i\lambda_1 2^{|l|}k_1\cdot (x-\frac{l}{\mu_1}t)}\Big{)},
   \end{align}
   then
   \begin{align}
   w_1=\sum\limits_{l\in Z^3}w_{1l},\qquad {\rm div}w_{1l}=0,\nonumber
   \end{align}
   and
   \begin{align}
   \hbox{div}w_1=0.\nonumber
   \end{align}
Moreover, if we set
\begin{align}
B_{1lk_1}:=&b_{1l}(x,t)B_{k_1}+\frac{1}{\lambda_1\lambda_{0}}\frac{\nabla b_{1l}(x,t)\times B_{k_1}}{2^{|l|}},\nonumber\\
B_{-1lk_1}:=&b_{1l}(x,t)B_{-k_1}+\frac{1}{\lambda_1\lambda_{0}}\frac{\nabla b_{1l}(x,t)\times B_{-k_1}}{2^{|l|}},\nonumber
\end{align}
then
\begin{align}
w_{1l}=B_{1lk_1}e^{i\lambda_1 2^{|l|} k_1\cdot (x-\frac{l}{\mu_1}t)}+B_{-1lk_1}e^{-i\lambda_1 2^{|l|} k_1\cdot (x-\frac{l}{\mu_1}t)}.\nonumber
\end{align}
Thus we complete the construction of perturbation $w_1$.

To construct $\chi_1$, we first
   denote $\beta_{1l}$ by
   \begin{align}\label{d:b1l}
    \beta_{1l}(x,t):=\frac{\alpha_l(\mu_1 v_0)}{2\sqrt{\bar{\rho}(t)}}\frac{g_{k_1}(-f_0(x,t))}{\gamma_{k_1}\big(\frac{R_0(x,t)}{\bar{\rho}(t)}\big)},
    \end{align}
    then denote the $l$-perturbation
    \begin{align}
    \chi_{1l}(x,t):=\beta_{1l}(x,t)\Big(e^{i\lambda_1 2^{|l|} k_1\cdot (x-\frac{l}{\mu_1}t)}+e^{-i\lambda_1 2^{|l|} k_1\cdot (x-\frac{l}{\mu_1}t)}\Big).
    \end{align}
 Finally, we define the perturbation
   \begin{align}
    \chi_1(x,t):=\sum_{l\in Z^3}\chi_{1l}.
   \end{align}
   Thus, $\chi_{1l}$ and $\chi_1$ are both real scalar functions, and as the perturbation of $w_1$, the summation in the definition of $\chi_1$ is meaningful.\\

   Moreover, since $\Big\|\frac{R_0(x,t)}{\bar{\rho}(t)}\Big\|_0\leq\frac{r_0}{2}$,  there exist two positive constants $c_{11}$ and $c_{21}$ such that $$c_{11}\leq\gamma_{k_1}\Big(\frac{R(x,t)}{\bar{\rho}(t)}\Big)\leq c_{21}.$$

  Thus, from (\ref{d:b1l}) on $\beta_{1l}$  and the assumption (\ref{e:reynold initial 2}) on $f_0$, after possibly taking a bigger number $M$ (still is a absolute constant only depend on $e$), we know that
  \begin{align}\label{b:bound 3}
  \|\beta_{1l}\|_0\leq \frac{M\sqrt{\delta}}{10L},
  \end{align}
  therefore, we have
  \begin{align}\label{e:bound x}
  \|\chi_1\|_0\leq \frac{M\sqrt{\delta}}{2L}.
  \end{align}

\subsection{The constructions of  $v_{01}$,~$p_{01}$,~$\theta_{01}$,~$f_{01}$,~$\mathring{R}_{01}$ }.

\indent

First, we denote
 $M_1$ by
   \begin{align}
   M_1=&\sum\limits_{l\in Z^3}b^2_{1l}(x,t)\Big{(}B_{k_1}\otimes B_{k_1}e^{2i\lambda_1 2^{|l|} k_1\cdot (x-\frac{l}{\mu_1}t)}+B_{-k_1}\otimes B_{-k_1}
   e^{-2i\lambda_1 2^{|l|} k_1\cdot (x-\frac{l}{\mu_1}t)}\Big{)}\nonumber\\
   &+\sum\limits_{l,l'\in Z^3 ,l\neq l'}w_{1ol}w_{1ol'}(x,t),
   \end{align}


 and
$N_1,K_1$ by
 \begin{align}
   N_1=&\sum_{l\in Z^3}\Big[w_{1l}\otimes \Big(v_0-\frac{l}{\mu_{1}}\Big)
   +\Big(v_0-\frac{l}{\mu_{1}}\Big{)}\otimes w_{1l}\Big]\nonumber\\
   K_1=&\sum\limits_{l\in Z^3}\beta_{1l}(x,t)b_{1l}(x,t)\Big(B_{k_1}e^{2i\lambda_1 2^{|l|} k_1\cdot (x-\frac{l}{\mu_1}t)}+B_{-k_1}e^{-2i\lambda_1 2^{|l|} k_1\cdot (x-\frac{l}{\mu_1}t)}\Big{)}+\sum\limits_{l,l'\in Z^3 ,l\neq l'}w_{1ol}\chi_{1l'}.
    \end{align}
 Notice that $N_{1}$ is a symmetric matrix.  Then we define
  \[ \begin{aligned}
    &v_{01}:=v_0+w_1,\\
    &p_{01}:=p_0-\frac{2w_{1o}\cdot w_{1oc}+|w_{1oc}|^2}{3}-tr(N_1),\\
    &\theta_{01}:=\theta_0+\chi_1-\fint_{T^{3}}{\chi_1}dx,\\
     &R_{01}:=-R_0(x,t)+2\sum\limits_{l\in Z^3}b^2_{1l}(x,t)Re(B_{k_1}\otimes B_{-{k_1}})+\delta \mathring{R}_{01},\\
    &f_{01}:=f_0+2\sum\limits_{l\in Z^3}\beta_{1l}(x,t)b_{1l}(x,t)A_{k_1}+\delta f_{01},
   \end{aligned}\]
   where
   \begin{align}\label{d:R0small}
   \delta \mathring{R}_{01}=&\mathcal{R}(\hbox{div}M_1)+N_1-tr(N_1)Id+\mathcal{R}\Big\{\partial_tw_{1}
  +\hbox{div}\Big[\sum_{l\in Z^3}\Big(w_{1l}\otimes\frac{l}{\mu_{1}}+\frac{l}{\mu_{1}}\otimes w_{1l}\Big)\Big]\Big\}\nonumber\\
   &+(w_{1o}\otimes w_{1oc}+w_{1oc}\otimes w_{1o}+w_{1oc}\otimes w_{1oc})\nonumber\\
   &-\frac{2w_{1o}\cdot w_{1oc}+|w_{1oc}|^2}{3}Id
   -\mathcal{R}\Big(\Big(\chi_1-\fint_{T^{3}}{\chi_1}dx\Big)e_3\Big),
   \end{align}
and
   \begin{align}\label{d:f0small}
   \delta f_{01}=&\mathcal{G}(\hbox{div}K_1)
   +\mathcal{G}(w_1\cdot\nabla\theta_{0})
   +\mathcal{G}\Big(\partial_t\chi_{1}+\sum_{l\in Z^3}\Big(\frac{l}{\mu_{1}}\cdot\nabla\Big)\chi_{1l}\Big)+\sum_{l\in Z^3}\Big(v_0-\frac{l}{\mu_1}\Big)\chi_{1l}+w_{1oc}\chi_1.
  \end{align}

From the property (\ref{l:R}) of $\mathcal{R}$, we know that $\delta \mathring{R}_{01}$ is a symmetric and trace-free matrix. Obviously,
   \begin{align*}
    \hbox{div}v_{1}=\hbox{div}v_0+\hbox{div}w_{1}=0.
   \end{align*}
   Moreover, by the definition of $w_{1o}, w_{1oc},\delta \mathring{R}_{01}$ as well as $v_{01},p_{01}$ and notice that $v_{0}, p_{0},
    \theta_0, \mathring{R}_{0}, f_{0}$  are solutions of the system (\ref{d:boussinesq reynold}),
together with Lemma \ref{l:R}, we know that
\[\begin{aligned}
    \hbox{div}R_{01}=&\hbox{div}\mathring{R_0}(x,t)+\partial_tw_1-\nabla(trN_1)
    -\nabla\Big(\frac{2w_{1o}\cdot w_{1oc}+|w_{1oc}|^2}{3}\Big)-\Big(\chi_1-\fint_{T^{3}}{\chi_1}dx\Big)e_3\nonumber\\
    &+\hbox{div}(w_{1o}\otimes w_{1o}+w_{1}\otimes v_{0}+v_0\otimes w_{1}
    +w_{1o}\otimes w_{1oc}+w_{1oc}\otimes w_{1o}+w_{1oc}\otimes w_{1oc})\nonumber\\
    =&\partial_tv_{0}+\hbox{div}(v_{0}\otimes v_{0})+\nabla p_{0}-\theta_0e_3+\partial_tw_1
    -\nabla(trN_1)-\nabla\Big(\frac{2w_{1o}\cdot w_{1oc}+|w_{1oc}|^2}{3}\Big)\nonumber\\
    &-\Big(\chi_1-\fint_{T^{3}}{\chi_1}dx\Big)e_3
    +\hbox{div}(w_{1o}\otimes w_{1o}+w_{1}\otimes v_{0}+v_0\otimes w_{1}
    +w_{1o}\otimes w_{1oc}\nonumber\\
    &+w_{1oc}\otimes w_{1o}+w_{1oc}\otimes w_{1oc})\nonumber\\
    =&\partial_tv_{01}+\hbox{div}(v_{01}\otimes v_{01})+\nabla p_{01}-\theta_{01}e_3.
    \end{aligned}\]
    Where we used $$\fint_{T^3}w_{1}(x,t)dx=0,$$
    and $$\hbox{div}(M_1)+\hbox{div}\Big(2\sum\limits_{l\in Z^3}b^2_{1l}(x,t)Re(B_{k_1}\otimes B_{-{k_1}})\Big)=\hbox{div}(w_{1o}\otimes w_{1o}).$$
  Furthermore, from the definition of $w_{1o}, \chi_1$, we know that
  \[ \begin{aligned}
    f_{01}=&f_0+2\sum\limits_{l\in Z^3}\beta_{1l}(x,t)b_{1l}(x,t)A_{k_1}+\mathcal{G}(\hbox{div}K_1)
   +\mathcal{G}(w_1\cdot\nabla\theta_{0})\nonumber\\
   &+\mathcal{G}\Big(\partial_t\chi_{1}+\sum_{l\in Z^3}\Big(\frac{l}{\mu_{1}}\cdot\nabla\Big)\chi_{1l}\Big)
   +\sum_{l\in Z^3}\Big(v_0-\frac{l}{\mu_1}\Big)\chi_{1l}+w_{1oc}\chi_1.\nonumber
   \end{aligned}\]
 Therefore, by the fact that $v_{0}, p_{0},
    \theta_0, \mathring{R}_{0}, f_{0}$ are solutions of  the system (\ref{d:boussinesq reynold})and by Lemma \ref{l:G inverse}
 \begin{align}
 \hbox{div}f_{01}=&\hbox{div}f_0+\partial_{t}\Big(\chi_1-\fint_{T^3}\chi_1dx\Big)
 +\hbox{div}(w_{1o}\chi_1+w_{1oc} \chi_1+v_0 \chi_1+w_1 \theta_0)\nonumber\\
 =&\hbox{div}(v_0\theta_0+w_{1o}\chi_1+w_{1oc} \chi_1+v_0 \chi_1+w_1 \theta_0)
 +\partial_t\Big(\theta_0+\chi_1-\fint_{T^3}\chi_1dx\Big)-h\nonumber\\
 =&\partial_t\theta_{01}+\hbox{div}(v_{01}\theta_{01})-h.\nonumber
 \end{align}
 Thus the functions $v_{01},p_{01},\theta_{01},R_{01},f_{01}$ solve the system (\ref{d:boussinesq reynold}).

\section{The representations}

\indent

In this section, we will calculate the form of
$$-R_0(x,t)+2\sum\limits_{l\in Z^3}b^2_{1l}(x,t)Re(B_{k_1}\otimes B_{-{k_1}}),$$
and
$$f_0+2\sum\limits_{l\in Z^3}\beta_{1l}(x,t)b_{1l}(x,t)A_{k_1}.$$
We will use the following basic fact.
\begin{proposition}
We have the following identity
\begin{align}\label{p:identity}
2Re(B_{k_1}\otimes B_{-{k_1}})=Id-\frac{k_1}{|k_1|}\otimes\frac{k_1}{|k_1|}.
\end{align}
\end{proposition}
\begin{proof}
We refer the proof to  \cite{CDL2}.
\end{proof}

\subsection{The representation of $-R_0(x,t)+2\sum\limits_{l\in Z^3}b^2_{1l}(x,t)Re(B_{k_1}\otimes B_{-{k_1}})$}

\indent

First, from the definition (\ref{d:l amp}) of $b_{1l}(x,t)$ and identity (\ref{p:identity}), we have
\begin{align}
2\sum\limits_{l\in Z^3}b^2_{1l}(x,t)Re(B_{k_1}\otimes B_{-{k_1}})
=&\bar{\rho}(t)\sum\limits_{l\in Z^3}\alpha_l^2(\mu_1 v_0)\gamma_{k_1}^2\Big{(}\frac{R_0(x,t)}{\bar{\rho}(t)}\Big{)}\Big(Id-\frac{k_1}{|k_1|}\otimes\frac{k_1}{|k_1|}\Big)\nonumber\\
=&\bar{\rho}(t)\gamma_{k_1}^2\Big{(}\frac{R_0(x,t)}{\bar{\rho}(t)}\Big{)}\Big(Id-\frac{k_1}{|k_1|}\otimes\frac{k_1}{|k_1|}\Big).\nonumber
\end{align}
where we used $\sum\limits_{l\in Z^3}\alpha_l^2=1$.\\
Moreover, from the identity (\ref{p:split 1}) and estimate (\ref{b:bound R 0}), we can decompose $R_{0}$
\begin{align}
\frac{R_0}{\bar{\rho}(t)}=\sum\limits_{i=1}^{L}\Big(\gamma_{k_i}\Big(\frac{R_0}{\bar{\rho}(t)}\Big)\Big)^2\Big(Id-\frac{k_i}{|k_i|}\otimes\frac{k_i}{|k_i|}\Big).\nonumber
\end{align}
Therefore
\begin{align}
-R_0(x,t)+2\sum\limits_{l\in Z^3}b^2_{1l}(x,t)Re(B_{k_1}\otimes B_{-{k_1}})
=-\bar{\rho}(t)\sum\limits_{i=2}^{L}\Big(\gamma_{k_i}\Big(\frac{R_0}{\bar{\rho}(t)}\Big)\Big)^2\Big(Id-\frac{k_i}{|k_i|}\otimes\frac{k_i}{|k_i|}\Big).\nonumber
\end{align}
Meanwhile, we have
\begin{align}
R_{01}=-\bar{\rho}(t)\sum\limits_{i=2}^{L}\Big(\gamma_{k_i}\Big(\frac{R_0}{\bar{\rho}(t)}\Big)\Big)^2\Big(Id-\frac{k_i}{|k_i|}\otimes\frac{k_i}{|k_i|}\Big)+\delta \mathring{R}_{01}.
\end{align}
Next section, we will prove that $\delta\mathring{R}_{01}$ is small.\\

\subsection{The representation of $f_0+2\sum\limits_{l\in Z^3}\beta_{1l}(x,t)b_{1l}(x,t)A_{k_1}$}

\indent

From the definition (\ref{d:l amp}) of $b_{1l}(x,t)$ and (\ref{d:b1l}) of $\beta_{1l}(x,t)$, we have
\begin{align}
2\sum\limits_{l\in Z^3}\beta_{1l}(x,t)b_{1l}(x,t)A_{k_1}
=\sum\limits_{l\in Z^3}\alpha_l^2(\mu_1 v_0)g_{k_1}(-f_0(x,t))A_{k_1}
=-g_{k_1}(f_0(x,t))A_{k_1},\nonumber
\end{align}
where we used the fact that $g_{k_1}$ is a linear operator.\\
Moreover, using the identity (\ref{p:split 2}), we can decompose $f_{0}$
\begin{align}
f_0(x,t)=\sum\limits_{i=1}^{3}g_{k_i}(f_0(x,t))A_{k_i},\nonumber
\end{align}
Therefore,
\begin{align}
&f_0+2\sum\limits_{l\in Z^3}\beta_{1l}(x,t)b_{1l}(x,t)A_{k_1}
=\sum\limits_{i=2}^{3}g_{k_i}(f_0(x,t))A_{k_i}\nonumber
\end{align}
Meanwhile, we have
\begin{align}
f_{01}=&f_0+2\sum\limits_{l\in Z^3}\alpha_l(\mu_1v_0)\beta_1(x,t)b_l(x,t)A_{k_1}+\delta f_{01}
=\sum\limits_{i=2}^{3}g_{k_i}(f_0(x,t))A_{k_i}+\delta f_{01}
\end{align}
Again in next section, we will prove that $\delta f_{01}$ is small.\\

 \section{Estimates on $\delta \mathring{R}_{01}$ and $\delta f_{01}$}

 \indent

First, we introduce some H\"{o}lder (semi)norm:
\begin{align*}
    \|f\|_{0}:=\hbox{sup}_{T^{3}\times [0,1]}|f|,\qquad [f]_{m}:=\hbox{max}_{|\beta|=m}\|D^{\beta}f\|_{0},\\
[f]_{m+\alpha}:=\hbox{max}_{|\beta|=m}\hbox{sup}_{x\neq y,t\in [0,1]}\frac{|D^{\beta}f(t,x)-D^{\beta}f(t,y)|}{|x-y|^{\alpha}},\\
\|f\|_{m}:=\sum_{0}^{m}[f]_{j},\|f\|_{m+\alpha}:=\|f\|_{m}+[f]_{m+\alpha}.
\end{align*}
where $ m=0,1,2,...$ and $\alpha\in (0,1)$, $\beta$ is a multi index.
 When $f(t,x)=f(x)$, the above H\"{o}lder(semi) norm denote usual spatial H\"{o}lder(semi) norm.
Moreover, in the subsequent estimate, unless otherwise stated, $\alpha\in (0,1)$  and $C_1$ denotes a constant which depends on $v_0,\mathring{R_0},e(t)$ as well as $\lambda_0,\alpha,\delta$,
  but does not depend on $\mu_1,\lambda_1$, and can change
  from line to line. Furthermore, we assume that $1\ll\mu_1\ll\lambda_1$.

 In the following, we frequently use the elementary inequalities
 \begin{align}
\|fg\|_{r}\leq& C\bigl(\|f\|_{r}\|g\|_0+\|g\|_{r}\|f\|_0\bigr),\label{i:inequality 1}\\
{[}f{]}_s \leq& C([f]_r+\|f\|_0).\label{i:inequality 2}
\end{align}
for any $r\geq0$ and $0\leq s<r$. These inequalities can be found in \cite{CDL2}.

We summarize the main properties of $b_{1l}$ and $\beta_{1l}$.
\begin{Lemma}\label{l:estimate amp}
For any $|l|\leq C_1\mu_1$ and $r\geq0$, we have
\begin{align}
\|b_{1l}\|_r\leq C_1\mu_1^r,\nonumber\\
\|\beta_{1l}\|_r\leq C_1\mu_1^r,\nonumber\\
\|(\partial_t+\frac{l}{\mu_1}\cdot\nabla)b_{1l}\|_r\leq C_1\mu_1^{r+1},\nonumber\\
\|(\partial_t+\frac{l}{\mu_1}\cdot\nabla)\beta_{1l}\|_r\leq C_1\mu_1^{r+1}.\nonumber
\end{align}
\begin{proof}
We only need to prove them when $r$ is integer, others can be derived by inequality (\ref{i:inequality 2}).  Recall that
 \begin{align}
    b_{1l}(x,t):=\sqrt{\bar{\rho}(t)}\alpha_l(\mu_1 v_0)\gamma_{k_1}\Big{(}\frac{R_0(x,t)}{\bar{\rho}(t)}\Big{)},\nonumber
   \end{align}
Then it is easy to check the first inequality in Lemma \ref{l:estimate amp}. Other three estimates are similar.
\end{proof}
\end{Lemma}

Then from the definition (\ref{d:w 1oc}) on $w_{10c}$, we have the following estimates
 \begin{Lemma}[Estimates on corrections]\label{e: estimate correction}
 \begin{align}
 \|w_{1oc}\|_\alpha\leq C_1\frac{\mu_1}{\lambda_1^{1-\alpha}},\qquad \forall \alpha\in [0,1)
 \end{align}
 \begin{proof}
 From the definition (\ref{d:w 1ocl}), inequality (\ref{i:inequality 1}) and Lemma \ref{l:estimate amp}
 \begin{align}
 \|w_{1ocl}\|_\alpha
 \leq C_1\Big(\frac{\|\nabla b_{1l}\|_\alpha}{\lambda_1}+\frac{\|\nabla b_{1l}\|_0}{\lambda_1^{1-\alpha}}\Big)
 \leq C_1\frac{\mu_1}{\lambda_1^{1-\alpha}}.\nonumber
 \end{align}
 Then, from the property (\ref{p:unity}) of $\alpha_l$, we have
 \begin{align}
 \|w_{1oc}\|_\alpha\leq C_1\frac{\mu_1}{\lambda_1^{1-\alpha}}.\nonumber
 \end{align}
 \end{proof}
 \end{Lemma}
 \subsection{Estimates on $\delta \mathring{R}_{01}$}

 \indent

 As in \cite{CDL2}, we also split the stresses into three parts, they are \\
   (1) the oscillation part $$\mathcal{R}(\hbox{div}M_1)-\mathcal{R}
   \Big(\Big(\chi_1-\fint_{T^{3}}{\chi_1}dx\Big)e_3\Big),$$
   (2) the transportation part
  \begin{align}
  &\mathcal{R}\Big\{\partial_tw_{1oc}+\partial_{t}w_{1o}
   +\hbox{div}\Big[\sum_{l\in Z^3}\Big(w_{1l}\otimes\frac{l}{\mu_{1}}+\frac{l}{\mu_{1}}\otimes w_{1l}\Big)\Big]\Big\}\nonumber\\
   =&\mathcal{R}\Big(\partial_tw_{1}+\sum_{l\in Z^3}\Big(\frac{l}{\mu_{1}}\cdot\nabla\Big)w_{1l}\Big),\nonumber
   \end{align}
   (3) the error part
  \begin{align}
  &N_1-tr(N_1)Id+(w_{1o}\otimes w_{1oc}+w_{1oc}\otimes w_{1o}-w_{1oc}\otimes w_{1oc})-\frac{2w_{1o}\cdot w_{1oc}+|w_{1oc}|^2}{3}Id.\nonumber
   \end{align}
In the following, we will estimate each part separately.

\begin{Lemma}[The oscillation part]\label{e:oscillation estimate}
\begin{align}
\|\mathcal{R}({\rm div}M_1)\|_\alpha\leq C_1\frac{\mu_1^3}{\lambda_1^{1-\alpha}},\label{e:oscillation estimate 1}\\
\Big\|\mathcal{R}\Big(\Big(\chi_1-\fint_{T^{3}}{\chi_1}dx\Big)e_3\Big)\Big\|_\alpha\leq C_1\frac{\mu_1}{\lambda_1^{1-\alpha}}.\label{e:oscillation estimate 2}
\end{align}
\begin{proof}
We start with the fact that $k_1\cdot B_{k_1}=0,$ then
\begin{align}
\hbox{div}M_1=&\sum\limits_{l\in Z^3}\Big{(}B_{k_1}\otimes B_{k_1}e^{2i\lambda_1 2^{|l|} k_1\cdot (x-\frac{l}{\mu}t)}+B_{-k_1}\otimes B_{-k_1}
   e^{-2i\lambda_1 2^{|l|} k_1\cdot (x-\frac{l}{\mu}t)}\Big{)}\nabla(b^2_{1l}(x,t))\nonumber\\
   &+\sum\limits_{l,l'\in Z^3 ,l\neq l'}\hbox{div}(w_{1ol}\otimes w_{1ol'})\nonumber\\
   =&M_{11}+M_{12}.\nonumber
\end{align}
By the Lemma \ref{l:estimate amp} and property (\ref{p:R G}) on $\mathcal{R}$,
\begin{align}
\|\mathcal{R}(\hbox{div}M_{11})\|_\alpha
\leq \sum\limits_{|l|\leq C_1\mu_1}C_1\frac{\mu_1}{\lambda_1^{1-\alpha}}
\leq C_1\frac{\mu_1^2}{\lambda_1^{1-\alpha}},\nonumber
\end{align}
where we used the following inequality which deduced from Lemma \ref{l:estimate amp}, for $\mu_1\ll\lambda_1$ and $m$ a fixed large number,
\begin{align}\label{e:R b1l}
&\Big\|\mathcal{R}\Big(\nabla b_{1l}e^{i\lambda_1 2^{|l|} k_1\cdot x}\Big)\Big\|_{\alpha}\nonumber\\ \leq&C_{1,m}\Big(\frac{\|\nabla b_{1l}\|_0}{\lambda_1^{1-\alpha}}+\frac{[\nabla b_{1l}]_m}{\lambda_1^{m-\alpha}}+\frac{[\nabla b_{1l}]_{m+\alpha}}{\lambda_1^m}\Big)\nonumber\\
\leq&C_1\frac{\mu_1}{\lambda_1^{1-\alpha}}
\end{align}
This fact will be used frequently in the subsequence estimates.\\

By direct calculations,
\begin{align}
M_{12}=&\sum\limits_{l,l'\in Z^3 ,l\neq l'}b_{1l}b_{1l'}\Big(B_{k_1}\otimes B_{k_1}e^{i\lambda_1(2^{|l|}+2^{|l'|})k_1\cdot x-ig_{1,l,l'}(t)}+B_{k_1}\otimes B_{-k_1}e^{i\lambda_1(2^{|l|}-2^{|l'|})k_1\cdot x-i\overline{g}_{1,l,l'}(t)}\nonumber\\
&+B_{-k_1}\otimes B_{k_1}e^{i\lambda_1(2^{|l'|}-2^{|l|})k_1\cdot x+i\overline{g}_{1,l,l'}(t)}+B_{-k_1}\otimes B_{-k_1}e^{-i\lambda_1(2^{|l|}+2^{|l'|})k_1\cdot x-ig_{1,l,l'}(t)}\Big),\nonumber
\end{align}
where $$g_{1,l,l'}(t)=\lambda_1\Big(2^{|l|}k_1\cdot\frac{l}{\mu_1}t+2^{|l'|}k_1\cdot\frac{l'}{\mu_1}t\Big),$$
 and $$\overline{g}_{1,l,l'}(t)=\lambda_1\Big(2^{|l|}k_1\cdot\frac{l}{\mu_1}t-
 2^{|l'|}k_1\cdot\frac{l'}{\mu_1}t\Big).$$
Since $k_1\cdot B_{k_1}=k_1\cdot B_{-k_1}=0$, we have
 \begin{align}
\hbox{div}M_{12}=&\sum\limits_{l,l'\in Z^3 ,l\neq l'}\Big(B_{k_1}\otimes B_{k_1}e^{i\lambda_1(2^{|l|}+2^{|l'|})k_1\cdot x-ig_{l,l'}(t)}+B_{k_1}\otimes B_{-k_1}e^{i\lambda_1(2^{|l|}-2^{|l'|})k_1\cdot x-i\overline{g}_{l,l'}(t)}\nonumber\\
&+B_{-k_1}\otimes B_{k_1}e^{i\lambda_1(2^{|l'|}-2^{|l|})k_1\cdot x+i\overline{g}_{l,l'}(t)}+B_{-k_1}\otimes B_{-k_1}e^{-i\lambda_1(2^{|l|}+2^{|l'|})k_1\cdot x-ig_{l,l'}(t)}\Big)\nabla(b_{1l}b_{1l'}),\nonumber
\end{align}
then using Lemma \ref{l:estimate amp} and property (\ref{p:R G}) on $\mathcal{R}$, we have
\begin{align}
\|\mathcal{R}(\hbox{div}M_{12})\|_\alpha
\leq \sum\limits_{|l|,|l'|\leq C_1\mu_1}C_1\frac{\mu_1}{\lambda_1^{1-\alpha}}
\leq C_1\frac{\mu_1^3}{\lambda_1^{1-\alpha}}.\nonumber
\end{align}
Then we obtain the estimate (\ref{e:oscillation estimate 1}).\\

From Lemma \ref{l:estimate amp} and the property (\ref{p:R G}) of $\mathcal{R}$, similarly we obtain
\begin{align}
\Big\|\mathcal{R}\Big(\Big(\chi_1-\fint_{T^{3}}{\chi_1}dx\Big)e_3\Big)\Big\|_\alpha
\leq \sum\limits_{|l|\leq C_1\mu_1}\frac{C_1}{\lambda_1^{1-\alpha}}
\leq C_1\frac{\mu_1}{\lambda_1^{1-\alpha}}.\nonumber
\end{align}
where we used $\mathcal{R}(f(t))=0$ which can be easily proved from the definition of $\mathcal{R}$ and then we complete the proof of this lemma.
\end{proof}
\end{Lemma}

\begin{Lemma}[The transportation part]\label{e:transport estimate}
\begin{align}
\Big\|\mathcal{R}\Big(\partial_tw_{1}+\sum_{l\in Z^3}\frac{l}{\mu_{1}}\cdot\nabla w_{1l}\Big)\Big\|_\alpha\leq C_1\frac{\mu_1^2}{\lambda_1^{1-\alpha}}
\end{align}
\begin{proof}
Recall that
\begin{align}
w_{1l}=B_{1lk_1}e^{i\lambda_1 2^{|l|} k_1\cdot (x-\frac{l}{\mu_1}t)}+B_{-1lk_1}e^{-i\lambda_1 2^{|l|} k_1\cdot (x-\frac{l}{\mu_1}t)}\nonumber
\end{align}
and $$w_1=\sum_{l\in Z^3}w_{1l}.$$
From the identity
$$\Big(\partial_t+\frac{l}{\mu_{1}}\cdot\nabla\Big)e^{\pm i\lambda_1 2^{|l|} k_1\cdot (x-\frac{l}{\mu_1}t)}=0,$$
we have
\begin{align}
&\partial_tw_{1}+\sum_{l\in Z^3}\frac{l}{\mu_{1}}\cdot\nabla w_{1l}\nonumber\\
=&\sum_{l\in Z^3}\Big(\Big{(}\partial_t+\frac{l}{\mu_{1}}\cdot\nabla\Big)B_{1lk_1}e^{i\lambda_1 2^{|l|} k_1\cdot (x-\frac{l}{\mu_1}t)}+\Big(\partial_t+\frac{l}{\mu_{1}}\cdot\nabla\Big)B_{-1lk_1}e^{-i\lambda_1 2^{|l|} k_1\cdot (x-\frac{l}{\mu_1}t)}\Big{)}.\nonumber
   \end{align}
By the property (\ref{p:R G}) of $\mathcal{R}$ and Lemma \ref{l:estimate amp}, we have
\begin{align}
\Big\|\mathcal{R}\Big(\partial_tw_{1}+\sum_{l\in Z^3}\frac{l}{\mu_{1}}\cdot\nabla w_{1l}\Big)\Big\|_\alpha
\leq \sum\limits_{|l|\leq C_1\mu_1}C_1\frac{\mu_1}{\lambda_1^{1-\alpha}}
\leq C_1\frac{\mu_1^2}{\lambda_1^{1-\alpha}}.\nonumber
\end{align}
We obtained the proof the Lemma \ref{e:transport estimate}.
\end{proof}
\end{Lemma}

\begin{Lemma}[Estimates on error part I]\label{e:error 1}
\begin{align}
\|N_1\|_0\leq \frac{C_1}{\mu_1}.
\end{align}
\end{Lemma}
\begin{proof}
From the definition of $N_{1}$, we have
\begin{align}
 N_1=&\sum_{l\in Z^3}\Big[w_{1l}\otimes \Big(v_0-\frac{l}{\mu_{1}}\Big)
   +\Big(v_0-\frac{l}{\mu_{1}}\Big{)}\otimes w_{1l}\Big{)}\Big]\nonumber\\
   =&\sum_{l\in Z^3}\Big[w_{1ol}\otimes \Big(v_0-\frac{l}{\mu_{1}}\Big)
   +\Big(v_0-\frac{l}{\mu_{1}}\Big{)}\otimes w_{1ol}\Big{)}\Big]\nonumber\\
   &+\sum_{l\in Z^3}\Big[w_{1ocl}\otimes \Big(v_0-\frac{l}{\mu_{1}}\Big)
   +\Big(v_0-\frac{l}{\mu_{1}}\Big{)}\otimes w_{1ocl}\Big{)}\Big]\nonumber\\
   =&N_{11}+N_{12}.\nonumber
\end{align}
For the term $N_{11}$, using the definition (\ref{d:w 1ol}) of $w_{1ol}$ , we have
\begin{align}
&\sum_{l\in Z^3}w_{1ol}\otimes \Big(v_0-\frac{l}{\mu_{1}}\Big)\nonumber\\
=&\sum_{l\in Z^3}b_{1l}(x,t)\Big(B_{k_1}e^{i\lambda_1 2^{|l|} k_1\cdot (x-\frac{l}{\mu_1}t)}\nonumber+B_{-{k_1}}e^{-i\lambda_1 2^{|l|} k_1\cdot (x-\frac{l}{\mu_1}t)}\Big)\otimes \Big(v_0-\frac{l}{\mu_{1}}\Big).\nonumber
\end{align}
Obviously, $b_{1l}(x,t)\neq0$ if and only if $|\mu_1 v_0-l|\leq1.$
By (\ref{p:unity}), it's easy to check
\begin{align}
\Big\|\sum_{l\in Z^3}w_{1ol}\otimes \Big(v_0-\frac{l}{\mu_{1}}\Big)\Big\|_0\leq \frac{C_1}{\mu_1}.\nonumber
\end{align}
Then, we have
\begin{align}
\|N_{11}\|_0\leq \frac{C_1}{\mu_1}.\nonumber
\end{align}
Similarly, by Lemma \ref{e: estimate correction}, we have
\begin{align}
\|N_{12}\|_0\leq \frac{C_1}{\lambda_1}.\nonumber
\end{align}
Then we obtain (6.66), in particular,
\begin{align}
\|trN_1\|_0\leq \frac{C_1}{\mu_1},\nonumber
\end{align}
and we complete our proof of this lemma.
\end{proof}

\begin{Lemma}[Estimates on error part II]\label{e: error 2}
\begin{align}
\|w_{1o}\otimes w_{1oc}+w_{1oc}\otimes w_{1o}+w_{1oc}\otimes w_{1oc}\|_0\leq C_1\frac{\mu_1}{\lambda_1}.
\end{align}
\begin{proof}
From the estimate (\ref{e:bound 1o}) and  Lemma \ref{e: estimate correction}, we have
\begin{align}
&\|w_{1o}\otimes w_{1oc}+w_{1oc}\otimes w_{1o}+w_{1oc}\otimes w_{10c}\|_0\nonumber\\
\leq& C_1(\|w_{1o}\|_0\|w_{1oc}\|_0+\|w_{1oc}\|_0^2)\nonumber\\
\leq& C_1\frac{\mu_1}{\lambda_1},\nonumber
\end{align}
therefore
\begin{align}
\|2w_{1o}\cdot w_{1oc}+|w_{1oc}|^2\|_0\leq C_1\frac{\mu_1}{\lambda_1}.\nonumber
\end{align}
\end{proof}
\end{Lemma}


Finally, by  Lemma \ref{e:oscillation estimate}, Lemma \ref{e:transport estimate}, Lemma \ref{e:error 1} and Lemma \ref{e: error 2}, we conclude that
\begin{align}
\|\delta\mathring{R}_{01}\|_0\leq C_1\Big(\frac{\mu_1^3}{\lambda_1^{1-\alpha}}+\frac{1}{\mu_1}\Big),\qquad \forall\alpha\in(0,1).
\end{align}
Furthermore, from the above estimates and constructions of $v_{01}, p_{01}$, we conclude that
\begin{align}
\|v_{01}-v_0\|_0\leq& \frac{M\sqrt{\delta}}{2L}+C_1\frac{\mu_1}{\lambda_1},\\
\|p_{01}-p_0\|_0\leq& C_1\Big(\frac{1}{\mu_1}+\frac{\mu_1}{\lambda_1} \Big).
\end{align}

\subsection{Estimates on $\delta f_{01}$}

 \indent

Recall (4.55), as before, we split $\delta f_{01}$ into three parts:\\
(1) the oscillation part
\begin{align}
\mathcal{G}(\hbox{div}K_1)+\mathcal{G}(w_1\cdot\nabla\theta_0),\nonumber
\end{align}
(2) the transportation part
\begin{align}
\mathcal{G}\Big(\partial_t\chi_{1}+\sum_{l\in Z^3}\Big(\frac{l}{\mu_{1}}\cdot\nabla\Big)\chi_{1l}\Big),\nonumber
\end{align}
(3) the error part
\begin{align}
w_{10c}\chi_1+\sum\limits_{l\in Z^3}\Big(v_0-\frac{l}{\mu}\Big)\chi_{1l}.\nonumber
\end{align}

\begin{Lemma}[The Oscillation Part]\label{e:osci}
\begin{align}
\|\mathcal{G}({\rm div}K_1)+\mathcal{G}(w_1\cdot\nabla\theta_0)\|_\alpha\leq C_1\frac{\mu_1^3}{\lambda_1^{1-\alpha}}.
\end{align}
\begin{proof}
From (4.53) and the definition of $K_1$, we may rewrite
\begin{align}
\hbox{div}K_1=&\sum\limits_{l\in Z^3}\nabla(\beta_{1l}(x,t)b_{1l}(x,t))\cdot\Big(B_{k_1}e^{2i\lambda_1 2^{|l|} k_1\cdot (x-\frac{l}{\mu_1}t)}+B_{-k_1}e^{-2i\lambda_1 2^{|l|} k_1\cdot (x-\frac{l}{\mu_1}t)}\Big{)}\nonumber\\
   &+\sum\limits_{l,l'\in Z^3 ,l\neq l'}\hbox{div}(w_{1ol}\chi_{1l'})\nonumber\\
   =&K_{11}+K_{12}.\nonumber
\end{align}
By Lemma \ref{l:estimate amp} and the property (\ref{p:R G}) of $\mathcal{G}$ , we have
\begin{align}
\|\mathcal{G}(K_{11})\|_\alpha
\leq \sum\limits_{|l|\leq C_1\mu_1}C_1\frac{\mu_1}{\lambda_1^{1-\alpha}}
\leq C_1\frac{\mu_1^2}{\lambda_1^{1-\alpha}}.\nonumber
\end{align}
Where we used the following inequality which deduced from Lemma \ref{l:estimate amp}, for $\mu_1\ll\lambda_1$ and $m$ a fixed large number,
\begin{align}\label{e:R b12}
&\Big\|\mathcal{R}\Big(\nabla \beta_{1l}e^{i\lambda_1 2^{|l|} k_1\cdot x}\Big)\Big\|_{\alpha}\nonumber\\ \leq&C_{1,m}\Big(\frac{\|\nabla \beta_{1l}\|_0}{\lambda^{1-\alpha}}+\frac{[\nabla \beta_{1l}]_m}{\lambda^{m-\alpha}}+\frac{[\nabla \beta_{1l}]_{m+\alpha}}{\lambda^m}\Big)\nonumber\\
\leq&C_1\frac{\mu_1}{\lambda_1^{1-\alpha}}.
\end{align}
Since $k_1\cdot B_{k_1}=k_1\cdot B_{-k_1}=0$, we have
\begin{align}
K_{12}=&\hbox{div}\Big\{\sum\limits_{l,l'\in Z^3 ,l\neq l'}b_{1l}\beta_{1l'}(x,t)(x,t)\Big{(}B_{k_1}e^{i\lambda_1 2^{|l|} k_1\cdot (x-\frac{l}{\mu_1}t)}\nonumber\\
&+B_{-k_1}e^{-i\lambda_1 2^{|l|} k_1\cdot (x-\frac{l}{\mu_1}t)}\Big{)}
\Big{(}e^{i\lambda_1 2^{|l'|} k_1\cdot (x-\frac{l'}{\mu_1}t)}+e^{-i\lambda_1 2^{|l'|} k_1\cdot (x-\frac{l'}{\mu_1}t)}\Big{)}\Big\}\nonumber\\
=&\sum\limits_{l,l'\in Z^3 ,l\neq l'}\Big{(}B_{k_1}\cdot\nabla(b_{1l}\beta_{1l'})e^{i\lambda_1(2^{|l|}+2^{|l'|})k_1\cdot x-ig_{1,l,l'}(t)}+B_{k_1}\cdot\nabla(b_{1l}\beta_{1l'})e^{i\lambda_1(2^{|l|}-2^{|l'|})k_1\cdot x-i\overline{g}_{1,l,l'}(t)}\nonumber\\
&+B_{-k_1}\cdot\nabla(b_{1l}\beta_{1l'})e^{i\lambda_1(2^{|l'|}-2^{|l|})k_1\cdot x+i\overline{g}_{1,l,l'}(t)}+B_{-k_1}\cdot\nabla(b_{1l}\beta_{1l'})e^{-i\lambda_1(2^{|l|}+
2^{|l'|})k_1\cdot x-ig_{1,l,l'}(t)}\Big),\nonumber
\end{align}
thus from Lemma \ref{l:estimate amp} and the property (\ref{p:R G}) of $\mathcal{G}$ , we have
\begin{align}
\|\mathcal{G}(K_{12})\|_\alpha
\leq \sum\limits_{|l|,|l'|\leq C_1\mu_1}C_1\frac{\mu_1}{\lambda_1^{1-\alpha}}
\leq C_1\frac{\mu_1^3}{\lambda_1^{1-\alpha}}.\nonumber
\end{align}
Then we complete the proof of this lemma.
\end{proof}
\end{Lemma}

\begin{Lemma}[The transportation Part]\label{e:trans}
\begin{align}
\Big\|\mathcal{G}\Big(\partial_t\chi_{1}+\sum_{l\in Z^3}\Big(\frac{l}{\mu_{1}}\cdot\nabla\Big)\chi_{1l}\Big)\Big\|_\alpha\leq C_1\frac{\mu_1^2}{\lambda_1^{1-\alpha}}.
\end{align}
\begin{proof}
From (4.48), (4.49) and the definition of $\chi_1$, we have
\begin{align}
&\partial_t\chi_{1}+\sum_{l\in Z^3}\Big(\frac{l}{\mu_{1}}\cdot\nabla\Big)\chi_{1l}\nonumber\\
=&\sum\limits_{l\in Z^3}\Big(\partial_t+\frac{l}{\mu_1}\cdot\nabla\Big)\beta_{1l}\Big(e^{i\lambda_1 2^{|l|} k_1\cdot (x-\frac{l}{\mu_1}t)}+e^{-i\lambda_1 2^{|l|} k_1\cdot (x-\frac{l}{\mu_1}t)}\Big).\nonumber
\end{align}
By Lemma \ref{l:estimate amp} and the property (\ref{p:R G}) of $\mathcal{G}$, we have
\begin{align}
\Big\|\mathcal{G}\Big(\partial_t\chi_{1}+\sum_{l\in Z^3}\Big(\frac{l}{\mu_{1}}\cdot\nabla\Big)\chi_{1l}\Big)\Big\|_\alpha\leq  \sum\limits_{|l| C_1\mu_1}C_1\frac{\mu_1}{\lambda_1^{1-\alpha}}
\leq C_1\frac{\mu_1^2}{\lambda_1^{1-\alpha}}.\nonumber
\end{align}

\end{proof}
\end{Lemma}

\begin{Lemma}[The error Part]\label{e:est 1}
\begin{align}
\Big\|w_{1oc}\chi_1+\sum_{l\in Z^3}\Big(v_0-\frac{l}{\mu_1}\Big)\chi_{1l}\Big\|_0\leq C_1\Big(\frac{\mu_1}{\lambda_1}+\frac{1}{\mu_1}\Big).
\end{align}
\begin{proof}
First, from  Lemma \ref{e: estimate correction}, we have
\begin{align}
\|w_{1oc}\chi_1\|_0\leq C_1\frac{\mu_1}{\lambda_1},\nonumber
\end{align}

\begin{align}
\sum_{l\in Z^3}\Big(v_0-\frac{l}{\mu_1}\Big)\chi_{1l}\nonumber
=\sum\limits_{l\in Z^3}\Big(v_0-\frac{l}{\mu_1}\Big)\beta_{1l}\Big{(}e^{i\lambda_1 2^{|l|} k_1\cdot (x-\frac{l}{\mu_1}t)}+e^{-i\lambda_1 2^{|l|} k_1\cdot (x-\frac{l}{\mu_1}t)}\Big{)}.\nonumber
\end{align}
Obviously, $\beta_{1l}(x,t)\neq0$ if and only if $|\mu_1 v_0-l|\leq1,$
therefore,
\begin{align}
\Big\|\sum_{l\in Z^3}\Big(v_0-\frac{l}{\mu_1}\Big)\chi_{1l}\Big\|_0\leq \frac{C_1}{\mu_1}.\nonumber
\end{align}

\end{proof}
\end{Lemma}
Then we conclude that
\begin{align}
\|\delta f_{01}\|_0\leq C_1\Big(\frac{\mu_1^3}{\lambda_1^{1-\alpha}}+\frac{1}{\mu_1}\Big).
\end{align}
Furthermore, from (\ref{e:bound x}), we have
\begin{align}\label{e:difference}
 \|\theta_{01}-\theta_{0}\|_0\leq& \frac{M\sqrt{\delta}}{2L}+C_1\frac{\mu_1^2}{\lambda_1}.
\end{align}
Where we used the fact
\begin{align}
\sup_t\Big|\int_{T^3}\chi_1dx\Big|
\leq \sum\limits_{|l|\leq C_1\mu_1}C_1\frac{\mu_1}{\lambda_1}
\leq C_1\frac{\mu_1^2}{\lambda_1}.\nonumber
\end{align}

In conclusion, we have
\begin{align}
 &R_{01}:=-\rho(t)\sum\limits_{i=2}^{L}\Big(\gamma_{k_i}\Big(\frac{R_0}{\rho(t)}\Big)\Big)^2\Big(Id-
 \frac{k_i}{|k_i|}\otimes\frac{k_i}{|k_i|}\Big)+\delta \mathring{R}_{01},\nonumber\\
    &f_{01}:=\sum\limits_{i=2}^{3}g_{k_i}(f_0(x,t))A_{k_i}+\delta f_{01},\nonumber
\end{align}
and the estimates
\begin{align}
 \|v_{01}-v_{0}\|_0\leq& \frac{M\sqrt{\delta}}{2L}+C_1\frac{\mu_1}{\lambda_1} ,\nonumber\\
\|p_{01}-p_{0}\|_0\leq& C_1\Big(\frac{1}{\mu_1}+\frac{\mu_1}{\lambda_1} \Big),\nonumber\\
 \|\theta_{01}-\theta_{0}\|_0\leq& \frac{M\sqrt{\delta}}{2L}+C_1\frac{\mu_1^2}{\lambda_1} ,\nonumber\\
\|\delta\mathring{R_{01}}\|_0\leq& C_1\Big(\frac{\mu_1^3}{\lambda_1^{1-\alpha}}+\frac{1}{\mu_1}\Big),\qquad \forall\alpha\in(0,1),\nonumber\\
\|\delta f_{01}\|_0\leq& C_1\Big(\frac{\mu_1^3}{\lambda_1^{1-\alpha}}+\frac{1}{\mu_1}\Big).\qquad \forall\alpha\in(0,1).\nonumber
\end{align}
Thus we complete the first step.

\section{Constructions of $v_{0n}, p_{0n}, \theta_{0n}, R_{0n}, f_{0n}$, $2\leq n\leq L$}

In this section, we assume $2\leq n \leq L$ and we will construct $v_{0n}, p_{0n}, \theta_{0n}, R_{0n}, f_{0n}$ by inductions.

 Suppose that for $1 \leq m < n\leq L$, we have constructed $v_{0m},~p_{0m},~\theta_{0m},~R_{0m},~f_{0m}$ and they are the solutions of
  the system (\ref{d:boussinesq reynold}). Furthermore, we have
 \begin{align}\label{f:R0m f0m}
 &R_{0m}:=-\bar{\rho}(t)\sum\limits_{i=m+1}^{L}\Big(\gamma_{k_i}\Big(\frac{R_0}{\bar{\rho}(t)}\Big)\Big)^2\Big(Id-\frac{k_i}{|k_i|}\otimes\frac{k_i}{|k_i|}\Big)
 +\sum\limits_{i=1}^{m}\delta \mathring{R}_{0i},\nonumber\\
    &f_{0m}:=
    \left \{
    \begin {array}{ll}
    \sum\limits_{i=m+1}^{3}g_{k_i}(f_0(x,t))A_{k_i}+\sum\limits_{i=1}^{m}\delta f_{0i}, \quad m=1,2\\
   \sum\limits_{i=1}^{m}\delta f_{0i},\qquad m\geq3,
    \end{array}
    \right.
 \end{align}
 and
 \begin{align}\label{e:esti 0m}
 \|v_{0m}-v_{0(m-1)}\|_0\leq& \frac{M\sqrt{\delta}}{2L}+C_m\frac{\mu_m}{\lambda_m} ,\nonumber\\
\|p_{0m}-p_{0(m-1)}\|_0\leq& C_m\Big(\frac{1}{\mu_m}+\frac{\mu_m}{\lambda_m} \Big),\nonumber\\
 \|\theta_{0m}-\theta_{0(m-1)}\|_0\leq& \frac{M\sqrt{\delta}}{2L}+C_m\frac{\mu_m^2}{\lambda_m} ,\nonumber\\
\|\delta\mathring{R}_{0m}\|_0\leq& C_m\Big(\frac{\mu_{m}^3}{\lambda_{m}^{1-\alpha}}+\frac{1}{\mu_{m}}\Big),\nonumber\\
\|\delta f_{0m}\|_0\leq& C_m\Big(\frac{\mu_m^3}{\lambda_{m}^{1-\alpha}}+\frac{1}{\mu_{m}}\Big).\qquad \forall\alpha\in(0,1).
\end{align}
where parameters $\mu_m$ and $\lambda_m$  will be chosen sufficiently large, which satisfies
\begin{align}
\lambda_m, \mu_m, \frac{\lambda_m}{\mu_m}\in N,
\end{align}
and the constant $C_m$ denotes a constant which  depends on $v_{0(m-1)},\mathring{R_0},e(t)$ as well as $\lambda_0,\alpha,\delta$, but doesn't depend on $\mu_m,\lambda_m$, and can
  change from line to line. \\

Next, we construct n-th steps by induction.

\subsection{The perturbations $w_{n}$ and $\chi_n$}

   \indent

   For $2\leq n \leq L$ and any $l\in Z^3$, we denote $ b_{nl}, B_{k_n}$ by
   \begin{align}
    b_{nl}(x,t):=&\sqrt{\rho(t)}\alpha_l(\mu_n v_{0(n-1)})\gamma_{k_n}\Big{(}\frac{R_0(x,t)}{\rho(t)}\Big{)},\nonumber\\
    B_{k_n}:=&A_{k_n}+i\frac{k_n}{|k_n|}\times A_{k_n}.\nonumber
   \end{align}
  And define $l$-perturbations
   \begin{align}
   w_{nol}:=b_{nl}(x,t)\Big{(}B_{k_n}e^{i\lambda_n 2^{|l|} k_n\cdot (x-\frac{l}{\mu_n}t)}+B_{-k_n}e^{-i\lambda_n 2^{|l|} k_n\cdot (x-\frac{l}{\mu_n}t)}\Big{)},\nonumber
   \end{align}
  where the parameters $\mu_n$ and $\lambda_n$  will be chosen sufficiently large, which satisfies the following conditions
\begin{align}
\lambda_n, \mu_n, \frac{\lambda_n}{\mu_n}\in N.
\end{align}
Then, we define the n-th perturbation
   \begin{align}
    w_{no}:=\sum_{l\in Z^3}w_{nol}.\nonumber
   \end{align}
   Similar as in the first step, the $l$-corrections is denoted by
   \begin{align}
   w_{nocl}:=&\frac{1}{\lambda_n\lambda_{0}}\Big{(}\frac{\nabla b_{nl}(x,t)\times B_{k_n}}{2^{|l|}}e^{i\lambda_n 2^{|l|} k_n\cdot (x-\frac{l}{\mu_n}t)}\nonumber+\frac{\nabla b_{nl}(x,t)\times B_{-{k_n}}}{2^{|l|}}e^{-i\lambda_n 2^{|l|} k_n\cdot (x-\frac{l}{\mu_n}t)}\Big{)},\nonumber
   \end{align}
   and the n-th correction
   \begin{align}
    w_{noc}:=\sum_{l\in Z^3}w_{nocl}.\nonumber
    \end{align}
  Finally, we define n-th  perturbation
   \begin{align*}
    w_n:=&w_{no}+w_{noc}.
   \end{align*}
  If we denote $w_{nl}$ by
   \begin{align}
   w_{nl}=&\frac{1}{\lambda_n\lambda_{0}}{\rm curl}\Big(\frac{b_{nl}(x,t)B_{k_n}}{2^{|l|}}e^{i\lambda_n 10^{|l|} k_n\cdot (x-\frac{l}{\mu_n}t)}\nonumber+\frac{b_{nl}(x,t)B_{-{k_n}}}{10^{|l|}}e^{-i\lambda_n 2^{|l|}k_n\cdot (x-\frac{l}{\mu_n}t)}\Big{)}\nonumber
   \end{align}
   then
   $$w_n=\sum_{l\in Z^3}w_{nl},\qquad \hbox{div}w_{nl}=0,\qquad \hbox{div}w_n=0, $$
   and they are all real vector functions.
Moreover,
\begin{align}\label{b:bound 2}
\|w_{n0}\|_0\leq\frac{M\sqrt{\delta}}{2L}.
\end{align}
We set
\begin{align}
B_{nlk_n}:=&b_{nl}(x,t)B_{k_n}+\frac{1}{\lambda_n\lambda_{0}}\frac{\nabla b_{nl}(x,t)\times B_{k_n}}{2^{|l|}},\nonumber\\
B_{-nlk_n}:=&b_{nl}(x,t)B_{-k_n}+\frac{1}{\lambda_n\lambda_{0}}\frac{\nabla b_{nl}(x,t)\times B_{-k_n}}{2^{|l|}},\nonumber
\end{align}
then
\begin{align}
w_{nl}=B_{nlk_n}e^{i\lambda_n 2^{|l|} k_n\cdot (x-\frac{l}{\mu_n}t)}+B_{-nlk_n}e^{-i\lambda_n 2^{|l|} k_n\cdot (x-\frac{l}{\mu_n}t)}.\nonumber
\end{align}
Then we complete the construction of n-th perturbation $w_n$.\\
To construct $\chi_n$, we denote $\beta_{nl}$ by
   \begin{align}
    \beta_{nl}(x,t):=\left \{
    \begin {array}{ll}
    \frac{\alpha_{l}(\mu_nv_{0(n-1)})}{2\sqrt{\rho(t)}}\frac{g_{k_n}(-f_0(x,t))}{\gamma_{k_n}\big(\frac{R(x,t)}{\rho(t)}\big)}, \qquad n=2,3\\
    0,\qquad n\geq4,
    \end{array}
    \right.
    \end{align}
    then define $l$-perturbations
    \begin{align}
    \chi_{nl}(x,t):=\left \{
    \begin {array}{ll}
    \beta_{nl}(x,t)\Big(e^{i\lambda_n 2^{|l|} k_n\cdot (x-\frac{l}{\mu_n}t)}+e^{-i\lambda_n 2^{|l|} k_n\cdot (x-\frac{l}{\mu_n}t)}\Big),\quad n=2,3\\
    0,\qquad n\geq4,
    \end{array}
    \right.
    \end{align}
   and the n-th perturbation
   \begin{align}
    \chi_n(x,t):=\sum\limits_{l\in Z^3}\chi_{nl}(x,t)
    =\left \{
    \begin {array}{ll}
    \sum\limits_{l\in Z^3}\chi_{nl},\quad n=2,3\\
    0,\qquad n\geq4.
    \end{array}
    \right.
   \end{align}
   Moreover, since $\Big\|\frac{R_0(x,t)}{\rho(t)}\Big\|_0\leq\frac{r_0}{2}$,  there exist two positive constants $c_{1n}$ and $c_{2n}$ such that $$c_{1n}\leq\gamma_{k_n}\Big(\frac{R_0(x,t)}{\rho(t)}\Big)\leq c_{2n}.$$
  Then, from (7.82) the definition of $\beta_{nl}$, we have
 \begin{align}
  \|\beta_{nl}\|_0\leq\frac{M\sqrt{\delta}}{10L},
\end{align}
\begin{align}
  \|\chi_{n}\|_0\leq\frac{M\sqrt{\delta}}{2L}.
\end{align}

\subsection{The constructions of  $v_{0n}$, $p_{0n}$, $\theta_{0n}$, $f_{0n}$, $\mathring{R}_{0n}$ }

\indent

First, we denote $M_n$ by
   \begin{align}
   M_n=&\sum\limits_{l\in Z^3}b^2_{nl}(x,t)\Big{(}B_{k_n}\otimes B_{k_n}e^{2i\lambda_n 2^{|l|} k_n\cdot (x-\frac{l}{\mu_n}t)}\nonumber+B_{-k_n}\otimes B_{-k_n}
   e^{-2i\lambda_n 2^{|l|} k_n\cdot (x-\frac{l}{\mu_n}t)}\Big{)}\nonumber\\
   &+\sum\limits_{l,l'\in Z^3 ,l\neq l'}w_{nol}\otimes w_{nol'}(x,t).
   \end{align}
 and
$N_n,K_n$ by
 \begin{align}
   N_n=&\sum_{l\in Z^3}\Big[w_{nl}\otimes \Big(v_{0(n-1)}-\frac{l}{\mu_{n}}\Big)
   +\Big(v_{0(n-1)}-\frac{l}{\mu_{n}}\Big{)}\otimes w_{nl}\Big]\nonumber
   \end{align}
   and
   \begin{align}
   K_n=
    \left \{
    \begin {array}{ll}
    \sum\limits_{l\in Z^3}\beta_{nl}(x,t)b_{nl}(x,t)\Big(B_{k_n}e^{2i\lambda_n 2^{|l|} k_n\cdot (x-\frac{l}{\mu_n}t)}
   + B_{-k_n}e^{-2i\lambda_n 2^{|l|} k_n\cdot (x-\frac{l}{\mu_n}t)}\Big{)}\\ \qquad +\sum\limits_{l,l'\in Z^3 ,l\neq l'}w_{nol}\chi_{nl'},\quad n=2,3\\ \\
   0,\quad n\geq4.
    \end{array}
    \right.
 \end{align}
 Notice that $N_{n}$ is a symmetric matrix. Then we define
   \begin{align}
    &v_{0n}:=v_{0(n-1)}+w_n,\nonumber\\
    &p_{0n}:=p_{0(n-1)}-\frac{2w_{no}\cdot w_{noc}+|w_{noc}|^2}{3}-tr(N_n),\nonumber\\ \nonumber\\
    &\theta_{0n}:=
    \left \{
    \begin {array}{ll}
    \theta_{0(n-1)}+\chi_n-\fint_{T^{3}}{\chi_n}dx, \quad n=2,3\nonumber\\
    \theta_{03}, \quad n\geq 4 \nonumber
    \end{array}
    \right.\\ \nonumber\\
     &R_{0n}:=R_{0(n-1)}(x,t)+2\sum\limits_{l\in Z^3}b^2_{nl}(x,t)Re(B_{k_n}\otimes B_{-{k_n}})+\delta \mathring{R}_{0n},\nonumber\\
    &f_{0n}:
    =\left \{
    \begin {array}{ll}
    f_{0(n-1)}+2\sum\limits_{l\in Z^3}\beta_{nl}b_{nl}(x,t)A_{k_n}+\delta f_{0n},\quad n=2,3\\
     f_{0(n-1)}+\delta f_{0n}, \quad n\geq 4.
    \end{array}
    \right.\label{d:f0n}
   \end{align}
   Where
   \begin{align}
   \delta \mathring{R}_{0n}:=&\mathcal{R}(\hbox{div}M_n)+N_n-tr(N_n)Id+\mathcal{R}\Big\{\partial_tw_{noc}+\partial_{t}w_{no}\nonumber\\
   &+\hbox{div}\Big[\sum_{l\in Z^3}\Big(w_{nl}\otimes\frac{l}{\mu_{n}}+\frac{l}{\mu_{n}}\otimes w_{nl}\Big)\Big]\Big\}
   -\mathcal{R}\Big(\Big(\chi_n-\fint_{T^{3}}{\chi_n}dx\Big)e_3\Big)\nonumber\\
   &+(w_{no}\otimes w_{noc}+w_{noc}\otimes w_{no}+w_{noc}\otimes w_{noc})-\frac{2w_{no}\cdot w_{noc}+|w_{noc}|^2}{3}Id,\nonumber
   \end{align}

   and
   \begin{align}
   \delta f_{0n}
   :=\left \{
    \begin {array}{ll}
      \mathcal{G}(\hbox{div}K_n)
   +\mathcal{G}(w_n\cdot\nabla\theta_{0(n-1)})
   +\mathcal{G}\Big(\partial_t\chi_{n}+\sum_{l\in Z^3}\Big(\frac{l}{\mu_{n}}\cdot\nabla\Big)\chi_{nl}\Big)\nonumber\\
   +\sum_{l\in Z^3}\Big(v_{0(n-1)}-\frac{l}{\mu_n}\Big)\chi_{nl}+w_{noc}\chi_n, \quad n=2,3\\ \\
   \mathcal{G}(w_{n}\cdot\nabla\theta_{03}),\quad n\geq4.
   \end{array}
    \right.
  \end{align}
   From Lemma \ref{l:reyn}, we know that $\delta \mathring{R}_{0n}$ is a symmetric and trace-free matrix. And
   \begin{align*}
    \hbox{div}v_{0n}=\hbox{div}v_{0(n-1)}+\hbox{div}w_{n}=0
   \end{align*}
  By definition,  $w_{no}$, $w_{noc}$, $\delta R_{0n}$ together with $v_{0n}$, $p_{0n}$ and  $v_{0(n-1)}$, $p_{0(n-1)}$, $
    \theta_{0(n-1)}$, $R_{0(n-1)}$, $f_{0(n-1)}$ are solutions of the system (2.2), we have

   \[\begin{aligned}
    \hbox{div}R_{0n}=&\hbox{div}R_{0(n-1)}(x,t)+\partial_tw_n-\nabla(trN_n)
    -\nabla\Big(\frac{2w_{no}\cdot w_{noc}+|w_{noc}|^2}{3}\Big)-
    \Big(\chi_n-\fint_{T^{3}}{\chi_n}dx\Big)e_3\nonumber\\
    &+\hbox{div}(w_{no}\otimes w_{no}+w_{n}\otimes v_{0(n-1)}+v_{0(n-1)}\otimes w_{n}
    +w_{no}\otimes w_{noc}+w_{noc}\otimes w_{no}+w_{noc}\otimes w_{noc})\nonumber\\
    =&\partial_tv_{0(n-1)}+\hbox{div}(v_{0(n-1)}\otimes v_{0(n-1)})+\nabla p_{0(n-1)}-\theta_{0(n-1)}e_3+\partial_tw_n
    -\nabla(trN_n)\nonumber\\
    &-\nabla\Big(\frac{2w_{no}\cdot w_{noc}+|w_{noc}|^2}{3}\Big)-
    \Big(\chi_n-\fint_{T^{3}}{\chi_n}dx\Big)e_3
    +\hbox{div}(w_{no}\otimes w_{no}+w_{n}\otimes v_{0(n-1)}\nonumber\\
    &+v_{0(n-1)}\otimes w_{n}
    +w_{no}\otimes w_{noc}+w_{noc}\otimes w_{no}+w_{noc}\otimes w_{noc})\nonumber\\
    =&\partial_tv_{0n}+\hbox{div}(v_{0n}\otimes v_{0n})+\nabla p_{0n}-\theta_{0n}e_3.
    \end{aligned}\]

    Where we used $$\fint_{T^3}w_{n}(x,t)dx=0,$$
    and $$\hbox{div}(M_n)+\hbox{div}\Big(2\sum\limits_{l\in Z^3}b^2_{nl}(x,t)Re(B_{k_n}\otimes B_{-{k_n}})\Big)=\hbox{div}(w_{no}\otimes w_{no}).$$

  Furthermore, from the definition of $w_{no}, \chi_n$, we know that
  \[ \begin{aligned}
  f_{0n}
  =\left \{
    \begin {array}{ll}
    f_{0(n-1)}+2\sum\limits_{l\in Z^3}\beta_{nl}(x,t)b_{nl}(x,t)A_{k_n}+\mathcal{G}(\hbox{div}K_n)\nonumber\\
   +\mathcal{G}(w_n\cdot\nabla\theta_{0(n-1)})
   +\mathcal{G}\Big(\partial_t\chi_{n}+\sum_{l\in Z^3}\Big(\frac{l}{\mu_{n}}\cdot\nabla\Big)\chi_{nl}\Big)\nonumber\\
   +\sum_{l\in Z^3}\Big(v_{0(n-1)}-\frac{l}{\mu_n}\Big)\chi_{nl}+w_{noc}\chi_n,\quad n=2,3\\
    f_{0(n-1)}+\mathcal{G}(w_{n}\cdot\nabla\theta_3), \quad n\geq4
    \end{array}
    \right.
   \end{aligned}\]
 Thus, for $n=2,3$, we have
 \begin{align}
 \hbox{div}f_{0n}=&\hbox{div}f_{0(n-1)}+\partial_{t}\Big(\chi_n-\fint_{T^3}\chi_n\Big)
 +\hbox{div}(w_{no}\chi_n+w_{noc} \chi_n+v_{0(n-1)} \chi_n+w_n \theta_{0(n-1)})\nonumber\\
 =&\hbox{div}(v_{0(n-1)}\theta_{0(n-1)}+w_{no}\chi_n+w_{noc} \chi_n+v_{0(n-1)} \chi_n\nonumber\\
 &+w_n\theta_{0(n-1)})+\partial_t\Big(\theta_{0(n-1)}+\chi_n-\fint_{T^3}\chi_ndx\Big)-h\nonumber\\
 =&\partial_t\theta_{0n}+\hbox{div}(v_{0n}\theta_{0n})-h.\nonumber
 \end{align}
 For $n\geq4$, we have
 \begin{align}
 \hbox{div}f_{0n}=\hbox{div}f_{0(n-1)}+w_{n}\cdot\nabla\theta_{03}
 =&\partial_t\theta_{0(n-1)}+\hbox{div}(v_{0(n-1)}\theta_{0(n-1)}+w_{n}\theta_{0(n-1)})-h\nonumber\\
 =&\partial_t\theta_{0n}+\hbox{div}(v_{0n}\theta_{0n})-h,\nonumber
 \end{align}
 where $\theta_{0(n-1)}=\theta_{03}$ for $n\geq 4$.\\

 So the functions $v_{0n},~p_{0n},~\theta_{0n},~R_{0n},~f_{0n}$ solve the system (2.2).

\section{The representations}

\indent

In this section, we will calculate the forms of
$$R_{0(n-1)}(x,t)+2\sum\limits_{l\in Z^3}b^2_{nl}(x,t)Re(B_{k_n}\otimes B_{-{k_n}}),$$
and
$$f_{0(n-1)}+2\sum\limits_{l\in Z^3}\beta_{nl}(x,t)(x,t)b_{nl}(x,t)A_{k_n} ,\quad n=2,3.$$
As in (\ref{p:identity}), we have the following identity:
\begin{align}
2Re(B_{k_n}\otimes B_{-{k_n}})=Id-\frac{k_n}{|k_n|}\otimes\frac{k_n}{|k_n|}.\nonumber
\end{align}

\subsection{The representation of $R_{0(n-1)}(x,t)+2\sum\limits_{l\in Z^3}b^2_{nl}(x,t)Re(B_{k_n}\otimes B_{-{k_n}})$}

\indent

First, by the notation of $b_{nl}(x,t)$, we have
\begin{align}
2\sum\limits_{l\in Z^3}b^2_{nl}(x,t)Re(B_{k_n}\otimes B_{-{k_n}})
=&\bar{\rho}(t)\sum\limits_{l\in Z^3}\alpha_l^2(\mu_n v_{(0(n-1)})\gamma_{k_n}^2\Big{(}\frac{R_0(x,t)}{\bar{\rho}(t)}\Big{)}\Big(Id-
\frac{k_n}{|k_n|}\otimes\frac{k_n}{|k_n|}\Big)\nonumber\\
=&\bar{\rho}(t)\gamma_{k_n}^2\Big{(}\frac{R_0(x,t)}{\bar{\rho}(t)}\Big{)}\Big(Id-
\frac{k_n}{|k_n|}\otimes\frac{k_n}{|k_n|}\Big).\nonumber
\end{align}
Moreover, by (\ref{f:R0m f0m}),
\begin{align}
R_{0(n-1)}:=-\bar{\rho}(t)\sum\limits_{i=n}^{L}\Big(\gamma_{k_i}\Big(\frac{R_0}{\bar{\rho}(t)}\Big)\Big)^2\Big(Id-
\frac{k_i}{|k_i|}\otimes\frac{k_i}{|k_i|}\Big)
+\sum\limits_{i=1}^{n-1}\delta \mathring{R}_{0i}.\nonumber
\end{align}
Therefore
\begin{align}
&R_{0(n-1)}(x,t)+2\sum\limits_{l\in Z^3}b^2_{nl}(x,t)Re(B_{k_n}\otimes B_{-{k_n}})\nonumber\\
=&-\bar{\rho}(t)\sum\limits_{i=n+1}^{L}\Big(\gamma_{k_i}\Big(\frac{R_0}{\bar{\rho}(t)}\Big)\Big)^2\Big(Id-
\frac{k_i}{|k_i|}\otimes\frac{k_i}{|k_i|}\Big)+\sum\limits_{i=1}^{n-1}\delta \mathring{R}_{0i}.\nonumber
\end{align}
Meanwhile, we have
\begin{align}
R_{0n}=-\bar{\rho}(t)\sum\limits_{i=n+1}^{L}\Big(\gamma_{k_i}\Big(\frac{R_0}{\bar{\rho}(t)}\Big)\Big)^2\Big(Id-\frac{k_i}{|k_i|}\otimes\frac{k_i}{|k_i|}\Big)
+\sum\limits_{i=1}^{n}\delta \mathring{R}_{0i}.
\end{align}
In particular,
\begin{align}
R_{0L}=\sum\limits_{i=1}^{L}\delta R_{01}.\nonumber
\end{align}
In next section, we will prove that $\delta R_{0n}$ is small.\\

\subsection{The representation of $f_{0(n-1)}+2\sum\limits_{l\in Z^3}\beta_{nl}(x,t)b_{nl}(x,t)A_{k_n},n=2,3$}

\indent

From the definition of $b_{nl}(x,t)$ and $\beta_{nl}(x,t)$, we have
\begin{align}
2\sum\limits_{l\in Z^3}\beta_{nl}(x,t)b_{nl}(x,t)A_{k_n}\nonumber
=&\sum\limits_{l\in Z^3}\alpha_l^2(\mu_nv_{0(n-1)})g_{k_n}(-f_{0(n-1)}(x,t))A_{k_n}
=-g_{k_n}(f_{0(n-1)}(x,t))A_{k_n}.\nonumber
\end{align}
Where we used the fact that $g_{k_n}$ is a linear functional.
Moreover, by (\ref{f:R0m f0m}), for $n=2,3$,
\begin{align}
f_{0(n-1)}:=\sum\limits_{i=n}^{3}g_{k_i}(f_0(x,t))A_{k_i}+\sum\limits_{i=1}^{n-1}\delta f_{0i}.\nonumber
\end{align}
Therefore
\begin{align}
f_{0(n-1)}+2\sum\limits_{l\in Z^3}\beta_{nl}(x,t)b_{nl}(x,t)A_{k_n}\nonumber
=&\sum\limits_{i=n+1}^{3}g_{k_i}(f_0(x,t))A_{k_i}+\sum\limits_{i=1}^{n-1}\delta f_{0i}.\nonumber
\end{align}
And
\begin{align}
f_{0n}
=\sum\limits_{i=n+1}^{3}g_{k_i}(f_0(x,t))A_{k_i}+\sum\limits_{i=1}^{n}\delta f_{0i}.
\end{align}
In particular,
\begin{align}
f_{03}
=\sum\limits_{i=1}^{3}\delta f_{0i}.\nonumber
\end{align}
And
\begin{align}
f_{0n}
=\sum\limits_{i=1}^{n}\delta f_{0i},\nonumber
\end{align}
for $n\geq 3$.

 \section{Estimates on $\delta \mathring{R}_{0n}$ and $\delta f_{0n}$}

 \indent

 In this section, as before, $C_n$ denotes a constant which depends on $v_{0(n-1)},~\mathring{R_0},~e(t)$ as well as $\lambda_0,~\alpha,~\delta$, but doesn't depend on $\mu_n,~\lambda_n$, and can
  change from line to line.  Furthermore, we will set $1\ll\mu_n\ll\lambda_n$.

 First, we summarize some of the estimates of $b_{nl}$ and $\beta_{nl}$.
\begin{Lemma}\label{e: estimate bnl}
For any $|l|\leq C_n\mu_n$ and $r\geq0$,
\begin{align}
\|b_{nl}\|_r\leq C_n\mu_n^r,\nonumber\\
\|\beta_{nl}\|_r\leq C_n\mu_n^r,\nonumber\\
\|(\partial_t+\frac{l}{\mu_n}\cdot\nabla)b_{nl}\|_r\leq C_n\mu_n^{r+1},\nonumber\\
\|(\partial_t+\frac{l}{\mu_n}\cdot\nabla)b_{nl}\|_r\leq C_n\mu_n^{r+1}.\nonumber
\end{align}
\begin{proof}
The proof is similar to that of Lemma 6.1, we omit it here.
\end{proof}
\end{Lemma}

 From the definition of $w_{noc}$ , we have the following estimates
 \begin{Lemma}[Estimates on the n-th correction]
 \begin{align}
 \|w_{noc}\|_\alpha\leq C_n\frac{\mu_n}{\lambda_n^{1-\alpha}},\qquad \forall \alpha\in [0,1)
 \end{align}
 \begin{proof}
 The proof is similar to that of Lemma \ref{e: estimate correction}, we omit it here.
 \end{proof}
 \end{Lemma}

 \subsection{Estimates on $\delta \mathring{R}_{0n}$}

 \indent

 As in (4.54), we have
  \begin{align}
   \delta R_{0n}=&\mathcal{R}(\hbox{div}M_n)+N_n-tr(N_n)Id+\mathcal{R}\Big\{\partial_tw_{n}
   +\hbox{div}\Big[\sum_{l\in Z^3}\Big(w_{nl}\otimes\frac{l}{\mu_{n}}+\frac{l}{\mu_{n}}\otimes w_{nl}\Big)\Big]\Big\}\nonumber\\
   &+N_n+(w_{no}\otimes w_{noc}+w_{noc}\otimes w_{no}+w_{noc}\otimes w_{noc})\nonumber\\
   &-\frac{2w_{no}\cdot w_{noc}+|w_{noc}|^2}{3}Id-\mathcal{R}\Big(\Big(\chi_n-
   \fint_{T^{3}}{\chi_n}dx\Big)e_3\Big).\nonumber
   \end{align}
  We again split the stresses into three parts,  \\
  (1) the oscillation part $$\mathcal{R}(\hbox{div}M_n)-\mathcal{R}\Big(\Big(\chi_n-\fint_{T^{3}}{\chi_n}dx\Big)e_3\Big),$$
  (2) the transportation part
  \begin{align}
  &\mathcal{R}\Big\{\partial_tw_{n}
   +\hbox{div}\Big[\sum_{l\in Z^3}\Big(w_{nl}\otimes\frac{l}{\mu_{n}}+\frac{l}{\mu_{n}}\otimes w_{nl}\Big)\Big]\Big\}=\mathcal{R}\Big(\partial_tw_{n}+\sum_{l\in Z^3}\Big(\frac{l}{\mu_{n}}\cdot\nabla\Big)w_{nl}\Big),\nonumber
   \end{align}
  (3) the error part
  \begin{align}
  &N_n-tr(N_n)Id+(w_{no}\otimes w_{noc}+w_{noc}\otimes w_{no}+w_{noc}\otimes w_{noc})-\frac{2w_{no}\cdot w_{noc}+|w_{noc}|^2}{3}Id.\nonumber
   \end{align}
As before, we will estimate each term separately.

\begin{Lemma}[The oscillation part]
\begin{align}
\|\mathcal{R}({\rm div}M_n)\|_\alpha\leq& C_n\frac{\mu_n^3}{\lambda_n^{1-\alpha}},\qquad \\
\Big\|\mathcal{R}\Big(\Big(\chi_n-\fint_{T^{3}}{\chi_n}dx\Big)e_3\Big)\Big\|_\alpha\leq& C_n\frac{\mu_n}{\lambda_n^{1-\alpha}}.
\end{align}
\end{Lemma}

\begin{Lemma}[The transportation part]
\begin{align}
\Big\|\mathcal{R}\Big(\partial_tw_{n}+\sum_{l\in Z^3}\frac{l}{\mu_{n}}\cdot\nabla w_{n}\Big)\Big\|_\alpha\leq C_n\frac{\mu_n^2}{\lambda_n^{1-\alpha}}.
\end{align}
\end{Lemma}

\begin{Lemma}[Estimates on error part I]
\begin{align}
\|N_n\|_0\leq \frac{C_n}{\mu_n}.
\end{align}
\end{Lemma}

\begin{Lemma}[Estimates on error part II]
\begin{align}
\|w_{no}\otimes w_{noc}+w_{noc}\otimes w_{no}+w_{noc}\otimes w_{noc}\|_0\leq C_n\frac{\mu_n}{\lambda_n}.
\end{align}
\end{Lemma}

The proof of the above four lemmas are similar to that of Lemma \ref{e:oscillation estimate}, Lemma \ref{e:transport estimate}, Lemma \ref{e:error 1} and Lemma \ref{e: error 2} respectively,  we omit it here.

For $n\geq4$, $\chi_n=0$, then $$\mathcal{R}\Big(\Big(\chi_n-\fint_{T^{3}}{\chi_n}dx\Big)e_3\Big)\Big)=0.$$

Finally we obtain
\begin{align}
\|\delta\mathring{R}_{0n}\|_0\leq C_n\Big(\frac{\mu_n^3}{\lambda_n^{1-\alpha}}+\frac{1}{\mu_n}\Big),\qquad \forall\alpha\in(0,1).
\end{align}
From the above estimates and constructions of $v_{0n}, p_{0n}$, we also conclude that
\begin{align}
\|v_{0n}-v_{0(n-1)}\|_0\leq& \frac{M\sqrt{\delta}}{2L}+C_n\frac{\mu_n}{\lambda_n} ,\\
\|p_{0n}-p_{0(n-1)}\|_0\leq& C_n\Big(\frac{1}{\mu_n}+\frac{\mu_n}{\lambda_n} \Big).
\end{align}

\subsection{Estimates on $\delta f_{0n}$}

 \indent

Recall that when $n=2,3$, we have
\begin{align}
   \delta f_{0n}=&\mathcal{G}(\hbox{div}K_n)
   +\mathcal{G}(w_n\cdot\nabla\theta_{0(n-1)})
   +\mathcal{G}\Big(\partial_t\chi_{n}+\sum_{l\in Z^3}\Big(\frac{l}{\mu_{1}}\cdot\nabla\Big)\chi_{nl}\Big)\nonumber\\
   &+\sum_{l\in Z^3}\Big(v_0-\frac{l}{\mu_n}\Big)\chi_{nl}+w_{noc}\chi_n.\nonumber
\end{align}
where we used the fact $\mathcal{G}(f(t))=0.$

As before, we split $\delta f_{0n}$ into three parts:\\
(1) the Oscillation Part:
\begin{align}
\mathcal{G}(\hbox{div}K_n)+\mathcal{G}(w_n\cdot\nabla\theta_{0(n-1)}).\nonumber
\end{align}
(2) the transportation part:
\begin{align}
\mathcal{G}\Big(\partial_t\chi_{n}+\sum_{l\in Z^3}\Big(\frac{l}{\mu_{1}}\cdot\nabla\Big)\chi_{nl}\Big).\nonumber
\end{align}
(3) the error part:
\begin{align}
w_{noc}\chi_n+\sum_{l\in Z^3}\Big(v_0-\frac{l}{\mu_n}\Big)\chi_{nl}.\nonumber
\end{align}
When $n\geq4$, we have
\begin{align}
   \delta f_{0n}=\mathcal{G}(w_{n}\cdot\nabla\theta_{03}).\nonumber
  \end{align}
  First, we estimate $\delta f_{0n}$ for $n=2,3$.

\begin{Lemma}[The Oscillation Part]
\begin{align}
\|\mathcal{G}({\rm div}K_n)+\mathcal{G}(w_n\cdot\nabla\theta_{0(n-1)})\|_\alpha\leq C_n\frac{\mu_n^3}{\lambda_n^{1-\alpha}}.
\end{align}
\end{Lemma}

\begin{Lemma}[The transportation Part]
\begin{align}
\Big\|\mathcal{G}\Big(\partial_t\chi_{n}+\sum_{l\in Z^3}\Big(\frac{l}{\mu_{n}}\cdot\nabla\Big)\chi_{nl}\Big)\Big\|_\alpha\leq C_n\frac{\mu_n^2}{\lambda_n^{1-\alpha}}.
\end{align}
\end{Lemma}

\begin{Lemma}[The error Part]
\begin{align}
\Big\|w_{1oc}\chi_n+\sum_{l\in Z^3}\Big(v_{0(n-1)}-\frac{l}{\mu_{n}}\Big)\chi_{nl}\Big\|_0\leq C_n\Big(\frac{\mu_n^2}{\lambda_n}+\frac{1}{\mu_n}\Big).
\end{align}
\end{Lemma}
The proof of the above four lemmas are similar to that of Lemma \ref{e:osci}, Lemma \ref{e:trans}, Lemma \ref{e:est 1}  respectively,  we omit it here.\\
Then we conclude that for $n=2,3$
\begin{align}
\|\delta f_{0n}\|_0\leq C_n\Big(\frac{\mu_n^3}{\lambda_n^{1-\alpha}}+\frac{1}{\mu_n}\Big).
\end{align}
For $n\geq4$, since $$\delta f_{0n}=\mathcal{G}(w_{n}\cdot\nabla\theta_{03}),$$
then
\begin{align}
\|\delta f_{0n}\|_\alpha\leq C_n\frac{\mu_n}{\lambda_n^{1-\alpha}}.\nonumber
\end{align}
Therefore, for any $2\leq n\leq L$, we have
\begin{align}
\|\delta f_{0n}\|_\alpha\leq C_n\Big(\frac{\mu_n^3}{\lambda_n^{1-\alpha}}+\frac{1}{\mu_n}\Big).
\end{align}
Furthermore, as in (\ref{e:difference}), we have
\begin{align}
 \|\theta_{0n}-\theta_{0(n-1)}\|_0\leq& \frac{M\sqrt{\delta}}{2L}+C_n\frac{\mu_n^2}{\lambda_n} .
\end{align}
Where we used
\begin{align}
\sup_t\Big|\int_{T^3}\chi_ndx\Big|
\leq \sum\limits_{|l|\leq C_n\mu_n}C_n\frac{\mu_n}{\lambda_n}
\leq C_n\frac{\mu_n^2}{\lambda_n}.\nonumber
\end{align}
As before, we have
\begin{align}\label{f:nform}
 &R_{0n}:=\bar{\rho}(t)\sum\limits_{i=n+1}^{L}\Big(\gamma_{k_i}\Big(\frac{R_0}{\bar{\rho}(t)}\Big)\Big)^2\Big(Id-
 \frac{k_i}{|k_i|}\otimes\frac{k_i}{|k_i|}\Big)
 +\sum\limits_{i=1}^{n}\delta \mathring{R}_{0i},\\
 &f_{0n}:=
    \left \{
    \begin {array}{ll}
   \sum\limits_{i=n+1}^{3}g_{k_i}(f_0(x,t))A_{k_i}+\sum\limits_{i=1}^{n}\delta f_{0i},\quad n=2,3\\
   \sum\limits_{i=1}^{n}\delta f_{0i},\quad n\geq4,
    \end{array}
    \right.
 \end{align}
and
 \begin{align}
 \|v_{0n}-v_{0(n-1)}\|_0\leq& \frac{M\sqrt{\delta}}{2L}+C_n\frac{\mu_n}{\lambda_n} ,\label{e:difference estimate 1}\\
\|p_{0n}-p_{0(n-1)}\|_0\leq& C_n\Big(\frac{1}{\mu_n}+\frac{\mu_n}{\lambda_n} \Big),\\
 \|v_{0n}-v_{0(n-1)}\|_0\leq& \frac{M\sqrt{\delta}}{2L}+C_n\frac{\mu_n^2}{\lambda_n},\\
\|\delta\mathring{R}_{0n}\|_0\leq& C_n\Big(\frac{\mu_{n}^3}{\lambda_{n}^{1-\alpha}}+\frac{1}{\mu_{n}}\Big),\\
\|\delta f_{0n}\|_0\leq& C_n\Big(\frac{\mu_n^3}{\lambda_{n}^{1-\alpha}}+\frac{1}{\mu_{n}}\Big).\label{e:difference estimate 2}\qquad \forall\alpha\in(0,1).
\end{align}

In particular, we conclude that
\begin{align}
 R_{0L}=&\sum\limits_{i=1}^{L}\delta \mathring{R}_{01},\label{d:final R}\\
 f_{0L}=&\sum\limits_{i=1}^{L}\delta f_{0i}.\label{d:final f}
     \end{align}
 We rewrite $R_{0L}$ as $\mathring{R}_{0L}$, then it's a symmetric and trace-free matrix.\\

\section{Energy Estimate}

\indent

In this section, we assume that $\lambda_n\ll \lambda_{n+1}$ and $C_n< C_{n+1}$, then the functions $v_{0(n-1)},~w_n$  oscillate very slowly compared with $w_{n+1}$.
\begin{Lemma}[Estimates on the energy]
\begin{align}\label{e:energy estimate}
\Big|e(t)\Big(1-\frac{\delta}{2}\Big)-\int_{T^3}|v_{0L}|^2(x,t)dx\Big|\leq  \sum\limits_{i=1}^{L}C_i\frac{\mu^3_i}{\lambda_i}.
\end{align}
\end{Lemma}
\begin{proof}
First we have
\begin{align}
&\Big|e(t)\Big(1-\frac{\delta}{2}\Big)-\int_{T^3}|v_{0L}|^2(x,t)dx\Big|\nonumber\\
\leq&\Big|e(t)\Big(1-\frac{\delta}{2}\Big)-\int_{T^3}\Big|v_0+\sum\limits_{i=1}^{L}(w_{i0}+
w_{i0c})\Big|^2(x,t)dx\Big|\nonumber\\
\leq&\Big|e(t)\Big(1-\frac{\delta}{2}\Big)-\int_{T^3}|v_0|^2(x,t)dx-
\sum\limits_{i=1}^{L}\int_{T^3}tr(w_{io}\otimes w_{io})dx\Big|\nonumber\\
&+\Big|2\sum\limits_{i=1}^{L}\int_{T^3}v_0\cdot w_{i}dx+\sum\limits_{1\leq i\neq j\leq L}\int_{T^3}tr(w_i\otimes w_j)dx\Big|\nonumber\\
&+\Big|\sum\limits_{i=1}^{L}\int_{T^3}tr(w_{io}\otimes w_{ioc}+w_{ioc}\otimes w_{io}+w_{ioc}\otimes w_{ioc})dx\Big|\nonumber\\
\leq&I+II+III.\nonumber
\end{align}
We estimate each terms separately.
From the stationary phase estimates (\ref{e:average}), we have
\begin{align}\label{e:energy estimate 2}
II\leq  \sum\limits_{i=1}^{L}C_i\frac{\mu^2_i}{\lambda_i}.
\end{align}
where we used the fact that for $i<j$, the oscillation of $w_j$ is faster than that of $w_i$ and then,
\begin{align}
\Big|\int_{T^3}w_i\otimes w_j(x,t)dx\Big|\leq C_j\frac{\mu_j^2}{\lambda_j}.\nonumber
\end{align}
From the smallness of $w_{ioc}$, we have
\begin{align}\label{e:energy estimate 3}
III\leq \sum\limits_{i=1}^{L}C_i\frac{\mu_i}{\lambda_i}.
\end{align}
To estimate $I$, we first compute the term $tr\Big(\sum\limits_{i=1}^{L}w_{i0}\otimes w_{i0}\Big)$.
\begin{align}
w_{i0}\otimes w_{i0}=&2\sum\limits_{l\in Z^3}b^2_{il}Re(B_{k_i}\otimes B_{-k_i})
+\sum\limits_{l\in Z^3}b^2_{il}(x,t)\Big{(}B_{k_i}\otimes B_{k_i}e^{2i\lambda_n 2^{|l|} k_i\cdot (x-\frac{l}{\mu_i}t)}\nonumber\\
&+B_{-k_i}\otimes B_{-k_i}
   e^{-2i\lambda_i 2^{|l|} k_i\cdot (x-\frac{l}{\mu_i}t)}\Big{)}
   +\sum\limits_{l,l'\in Z^3 ,l\neq l'}b_{il}b_{il'}(x,t)\Big(B_{k_i}e^{i\lambda_i 2^{|l|} k_i\cdot (x-\frac{l}{\mu_i}t)}\nonumber\\
   &+B_{-k_i}e^{-i\lambda_i 2^{|l|} k_i\cdot (x-\frac{l}{\mu_i}t)}\Big)\otimes\Big(B_{k_i}e^{i\lambda_i 2^{|l'|} k_i\cdot (x-\frac{l'}{\mu_i}t)}\nonumber+B_{-k_i}e^{-i\lambda_i 2^{|l'|} k_i\cdot (x-\frac{l'}{\mu_i}t)}\Big)\nonumber\\
   =&2\sum\limits_{l\in Z^3}b^2_{il}Re(B_{k_i}\otimes B_{-k_i})+G_i\nonumber\\
   =&\bar{\rho}(t)\gamma^2_{k_i}\Big(\frac{R_0(x,t)}{\bar{\rho}(t)}\Big)\Big(Id-
   \frac{k_i}{|k_i|}\otimes\frac{k_i}{|k_i|}\Big)+G_i,\nonumber
\end{align}
where
\begin{align}
G_i=&\sum\limits_{l\in Z^3}b^2_{il}(x,t)\Big{(}B_{k_i}\otimes B_{k_i}e^{2i\lambda_i 2^{|l|} k_i\cdot (x-\frac{l}{\mu_i}t)}\nonumber+B_{-k_i}\otimes B_{-k_i}
   e^{-2i\lambda_i 2^{|l|} k_i\cdot (x-\frac{l}{\mu_i}t)}\Big{)}\nonumber\\
   &+\sum\limits_{l,l'\in Z^3 ,l\neq l'}b_{il}b_{il'}\Big(B_{k_i}\otimes B_{k_i}e^{i\lambda_1(2^{|l|}+2^{|l'|})k_1\cdot x-ig_{1,l,l'}(t)}\nonumber+B_{k_i}\otimes B_{-k_i}e^{i\lambda_1(2^{|l|}-2^{|l'|})k_i\cdot x-i\overline{g}_{1,l,l'}(t)}\nonumber\\
&+B_{-k_i}\otimes B_{k_i}e^{i\lambda_1(2^{|l'|}-2^{|l|})k_i\cdot x+i\overline{g}_{1,l,l'}(t)}\nonumber+B_{-k_i}\otimes B_{-k_i}e^{-i\lambda_1(2^{|l|}+2^{|l'|})k_i\cdot x-ig_{1,l,l'}(t)}\Big),\nonumber
  \end{align}
 is the oscillatory part.\\

Then, from (\ref{p:split 1}) and $tr\mathring{R_0}=0$, we have
\begin{align}
&tr\Big(\sum\limits_{i=1}^{L}w_{i0}\otimes w_{i0}\Big)\nonumber\\
=&tr\Big(\sum\limits_{i=1}^{L}\bar{\rho}(t)\gamma^2_{k_i}\Big(\frac{R_0(x,t)}{\bar{\rho}(t)}\Big)\Big(Id-\frac{k_i}{|k_i|}\otimes\frac{k_i}{|k_i|}\Big)
+tr\Big(\sum\limits_{i=1}^{L}G_i\Big)\nonumber\\
=&3\rho(t)+tr\Big(\sum\limits_{i=1}^{L}G_i\Big)\nonumber
\end{align}
From the definition (\ref{d:amp 2}) of $\bar{\rho}(t)$, we have
\begin{align}\label{e:energy estimate 1}
I=\Big|\int_{T^3}tr\Big(\sum\limits_{i=1}^{L}G_i\Big)dx\Big| .\nonumber
\end{align}
By Lemma \ref{l:estimate amp} and stationary phase estimate (\ref{e:average}), we have
\begin{align}
I\leq \sum\limits_{i=1}^{L}C_i\frac{\mu^3_i}{\lambda_i} .
\end{align}
Finally, from the estimates (\ref{e:energy estimate 1}), (\ref{e:energy estimate 2}), and (\ref{e:energy estimate 3}), we conclude that
\begin{align}
\Big|e(t)\Big(1-\frac{\delta}{2}\Big)-\int_{T^3}|v_{0L}(x,t)|^2dx\Big|\leq \sum\limits_{i=1}^{L}C_i\frac{\mu^3_i}{\lambda_i}.\nonumber
\end{align}
\end{proof}

\section{Proof of proposition 2.1}
In this section, we collect the estimates from the preceding sections. From these estimates , we can prove Proposition \ref{p: iterative 1} by choosing the appropriate parameters
$\mu_n, \lambda_n$ for $1\leq n\leq L$.
\begin{proof}
By the definition (\ref{d:final R}), (\ref{d:final f}) and the estimates (\ref{e:difference estimate 1})-(\ref{e:difference estimate 2}), (\ref{e:energy estimate}), we have $v_{0N},~p_{0N},~\theta_{0N},~\mathring{R}_{0N},~f_{0N}$, which solve system (\ref{d:boussinesq reynold}) and satisfy
\begin{align}
\|\mathring{R}_{0L}\|_0\leq& \sum\limits_{n=1}^{L}C_n\Big(\frac{\mu_n^3}{\lambda_n^{1-\alpha}}+\frac{1}{\mu_n}\Big),\nonumber\\
\|f_{0L}\|_0\leq& \sum\limits_{n=1}^{L}C_n\Big(\frac{\mu_n^3}{\lambda_n^{1-\alpha}}+\frac{1}{\mu_n}\Big),\qquad \forall\alpha\in(0,1).\nonumber\\
\|v_{0L}-v_0\|_0\leq&\frac{M\sqrt{\delta}}{2}+\sum\limits_{n=1}^{L}C_n\frac{\mu_n}{\lambda_n},\nonumber\\
\|p_{0L}-p_0\|_0\leq&\sum\limits_{n=1}^{L}C_n\Big(\frac{\mu_n}{\lambda_n}+\frac{1}{\mu_n}\Big),\nonumber\\
\|\theta_{0L}-\theta_0\|_0\leq&\frac{M\sqrt{\delta}}{2}+\sum\limits_{n=1}^{L}C_n\frac{\mu_n^2}{\lambda_n},\nonumber\\
\Big|e(t)\Big(1-\frac{\delta}{2}\Big)-&\int_{T^3}|v_{0L}(x,t)|^2dx\Big|\leq \sum\limits_{n=1}^{L}C_n\frac{\mu^3_n}{\lambda_n}.\nonumber
\end{align}
We let $\mu_n=\sqrt[5]\lambda_n$ , then we have
\begin{align}
\|\mathring{R}_{0L}\|_0\leq& \sum\limits_{n=1}^{L}C_n\Big(\frac{1}{\lambda_n^{\frac{2}{5}-\alpha}}+\frac{1}{\sqrt[5]\lambda_n}\Big),\nonumber\\
\|f_{0L}\|_0\leq& \sum\limits_{n=1}^{L}C_n\Big(\frac{1}{\lambda_n^{\frac{2}{5}-\alpha}}+\frac{1}{\sqrt[5]\lambda_n}\Big),\qquad \forall\alpha\in(0,1).\nonumber
\end{align}
and
\begin{align}
\|v_{0L}-v_0\|_0\leq&\frac{M\sqrt{\delta}}{2}+\sum\limits_{n=1}^{L}\frac{C_n}{\lambda_n^{\frac{4}{5}}},\nonumber\\
\|p_{0L}-p_0\|_0\leq&\sum\limits_{n=1}^{L}C_n\Big(\frac{1}{\lambda_n^{\frac{4}{5}}}+\frac{1}{\sqrt[5]\lambda_n}\Big),\nonumber\\
\|\theta_{0L}-\theta_0\|_0\leq&\frac{M\sqrt{\delta}}{2}+\sum\limits_{n=1}^{L}\frac{C_n}{\lambda_n^{\frac{3}{5}}},\nonumber\\
\Big|e(t)\Big(1-\frac{\delta}{2}\Big)-&\int_{T^3}|v_{0L}|^2(x,t)dx\Big|\leq  \sum\limits_{n=1}^{L}\frac{C_n}{\lambda_n^{\frac{2}{5}}}.\nonumber
\end{align}
Let $\lambda_1\ll\lambda_2\ll\cdot\cdot\cdot\ll\lambda_N$ sufficiently large, such that
 $$\frac{C_n}{\sqrt[5]\lambda_n}\leq \hbox{min}\Big(\frac{\eta\delta}{10L},\frac{M\sqrt{\delta}}{10L}, \frac{\delta}{8L}{\rm min}_te(t)\Big).$$
 Set $\alpha<\frac{1}{5}$, we conclude that
\begin{align}
\|\mathring{R}_{0L}\|_0\leq& \frac{\eta\delta}{2},\nonumber\\
\|f_{0L}\|_0\leq& \frac{\eta\delta}{2},\nonumber\\
\|v_{0L}-v_0\|_0\leq& M\sqrt{\delta},\nonumber\\
\|p_{0L}-p_0\|_0\leq& M\sqrt{\delta},\nonumber\\
\|\theta_{0L}-\theta_0\|_0\leq& M\sqrt{\delta},\nonumber\\
\Big|e(t)\Big(1-\frac{\delta}{2}\Big)-&\int_{T^3}|v_{0L}|^2(x,t)dx\Big|\leq \frac{\delta}{8}e(t).\nonumber
\end{align}
Finally, we set $$\tilde{v}=v_{0L},~\tilde{p}=p_{0L},~\tilde{\theta}=\theta_{0L},~\mathring{\tilde{R}}=\mathring{R}_{0L},~\tilde{f}=f_{0L}.$$
then $\tilde{v},~\tilde{p},~\tilde{\theta},~\mathring{\tilde{R}},~\tilde{f}$ are what we need in our Proposition (\ref{p: iterative 1}).
\end{proof}

\section{Proof of proposition 1.2}
In this section, we prove Proposition \ref{p:iterative 2}. The constructions of $\tilde{v},~\tilde{p},~\tilde{\theta},~\mathring{\tilde{R}},~\tilde{f}$ are similar
to the construction in Propositions \ref{p: iterative 1}. Again, we need multi-steps iterations. We only give the first step, the others could be constructed by inductions as in proposition \ref{p: iterative 1}. For convenience, we use the same notations as in Proposition \ref{p: iterative 1}.
\subsection{The main perturbation $w_{1o}$}

\indent

The construction is very similar to that of given in subsection 4.1, beside the choice of amplitude.
First, the functions $\alpha_l, \gamma_{k_1}, g_{k_1}$ are same as before.
We set
   \begin{align*}
        R_0(x,t):=\delta Id-\mathring{R}_0(x,t).
    \end{align*}

   For any $l\in Z^3$, we set
   \begin{align}
    b_{1l}(x,t):=&\sqrt{\delta}\alpha_l(\mu_1 v_0)\gamma_{k_1}\Big{(}\frac{R(x,t)}{\delta}\Big{)},\\
    B_{k_1}:=&A_{k_1}+i\frac{k_1}{|k_1|}\times A_{k_{1}},
   \end{align}
   and $l$-perturbations
   \begin{align}
    w_{1ol}:=b_{1l}(x,t)\Big{(}B_{k_1}e^{i\lambda_1 2^{|l|} k_1\cdot (x-\frac{l}{\mu_1}t)}+B_{-k_1}e^{-i\lambda_1 2^{|l|} k_1\cdot (x-\frac{l}{\mu_1}t)}\Big{)}.
   \end{align}
   Finally, we define 1-th perturbation
   \begin{align}
    w_{1o}:=\sum_{l\in Z^3}w_{1ol}.
   \end{align}
   Obviously, $w_{1ol},w_{1o}$ are all real 3-dimensional vector function.\\

   We remark that  $\hbox{supp} \alpha_l\cap \hbox{supp}\alpha_{l'}=\emptyset$  if $|l-l'|\geq2$, thus the above summation is meaningful.

 \subsection{The constants $\eta$ and $M$}

 \indent

   First, $b_{1l}$ is well-defined only if $\frac{R}{\delta}\in B_{r_{0}}(Id)$ where $r_{0}$ is given in geometric Lemma \ref{p:split}. However, since
   \begin{align}\label{b:near id}
    \Big{\|}\frac{\mathring{R}_0}{\delta}-Id\Big\|_0\leq \frac{\|\mathring{R}_0\|_0}{\delta }\leq \eta,
   \end{align}
 so it suffices to set
   \begin{align*}
    \eta:=\frac{1}{2}r_{0}.
   \end{align*}

   Next notice that there exist two positive constants $c_{10}$ and $c_{20}$ such that
   \begin{align*}
    c_{10}\leq\gamma_{k_1}\Big{(}\frac{R(x,t)}{\delta}\Big{)}\leq c_{20}.
   \end{align*}
   Thus there exists an absolute constant $ M>1$ such that
   \begin{align}
    \|b_{1l}\|_{0}\leq\frac{M\sqrt{\delta}}{10L}.\nonumber
   \end{align}
   Moreover, from the constructions of $\alpha_l$, we know that
     \begin{align}
    \|w_{1o}\|_{0}\leq\frac{M\sqrt{\delta}}{2L}.
   \end{align}

 \subsection{The correction $w_{1oc}$ and the perturbation $w_{1}$ and $\chi_1$}

   \indent

   The following constructions are same as in subsection 4.3. We define the $l$-corrections
   \begin{align}
   w_{1ocl}:=&\frac{1}{\lambda_1\lambda_{0}}\Big{(}\frac{\nabla b_{1l}(x,t)\times B_{k_1}}{2^{|l|}}e^{i\lambda_1 2^{|l|} k_1\cdot (x-\frac{l}{\mu_1}t)}+\frac{\nabla b_{1l}(x,t)\times B_{-{k_1}}}{2^{|l|}}e^{-i\lambda_1 2^{|l|} k_1\cdot (x-\frac{l}{\mu_1}t)}\Big{)},
   \end{align}
   then define 1-th correction
   \begin{align}
    w_{1oc}:=\sum_{l\in Z^3}w_{1ocl},
    \end{align}

   Finally, we define 1-th perturbation
   \begin{align}
    w_1:=w_{1o}+w_{1oc}
   \end{align}
  Thus, if we denote $w_{1l}$ by
   \begin{align}
   w_{1l}:=&w_{1ol}+w_{1ocl}\nonumber\\
   =&\frac{1}{\lambda_1\lambda_{0}}\hbox{curl}\Big(\frac{b_{1l}(x,t)B_{k_1}}{2^{|l|}}e^{i\lambda_1 2^{|l|} k_1\cdot (x-\frac{l}{\mu_1}t)}+\frac{b_{1l}(x,t)B_{-{k_1}}}{2^{|l|}}e^{-i\lambda_1 2^{|l|}k_1\cdot (x-\frac{l}{\mu_1}t)}\Big{)},
   \end{align}
   then
   \begin{align}
   w_1=\sum\limits_{l\in Z^3}w_{1l},\qquad \hbox{div}w_{1l}=0, \qquad {\rm div}w_1=0.\nonumber
   \end{align}
Moreover, if we set
\begin{align}
B_{1lk_1}=&b_{1l}(x,t)B_{k_1}+\frac{1}{\lambda_1\lambda_{0}}\frac{\nabla b_{1l}(x,t)\times B_{k_1}}{2^{|l|}},\nonumber\\
B_{-1lk_1}=&b_{1l}(x,t)B_{-k_1}+\frac{1}{\lambda_1\lambda_{0}}\frac{\nabla b_{1l}(x,t)\times B_{-k_1}}{2^{|l|}},\nonumber
\end{align}
then
\begin{align}
w_{1l}=B_{1lk_1}e^{i\lambda_1 2^{|l|} k_1\cdot (x-\frac{l}{\mu_1}t)}+B_{-1lk_1}e^{-i\lambda_1 2^{|l|} k_1\cdot (x-\frac{l}{\mu_1}t)}.\nonumber
\end{align}
Thus, we complete the constructions of perturbation $w_1$.

To construct $\chi_1$, we first
   denote $\beta_{1l}$ by
   \begin{align}
    \beta_{1l}(x,t):=\frac{\alpha_l(\mu_1 v_0)}{2\sqrt{\delta}}\frac{g_{k_1}(-f_0(x,t))}{\gamma_{k_1}\big(\frac{R_0(x,t)}{\delta}\big)},
    \end{align}
    then define $l$-perturbations
    \begin{align}
    \chi_{1l}(x,t):=\beta_{1l}(x,t)\Big(e^{i\lambda_1 2^{|l|} k_1\cdot (x-\frac{l}{\mu_1}t)}+e^{-i\lambda_1 2^{|l|} k_1\cdot (x-\frac{l}{\mu_1}t)}\Big).
    \end{align}
 Finally, we define the 1-th perturbation
   \begin{align}
    \chi_1(x,t):=\sum_{l\in Z^3}\chi_{1l}.
   \end{align}
   Thus, $\chi_{1l}$ and $\chi_1$ are both real scalar functions, and as the perturbation $w_1$, the summation in the definition of $\chi_1$ is meaningful.


  Thus, from the assumptions of proposition \ref{p:iterative 2}, we know that
  $$\|\beta_{1l}\|_0\leq \frac{M\sqrt{\delta}}{10L},$$
  therefore, we have
  $$\|\chi_1\|_0\leq\frac{M\sqrt{\delta}}{2L}.$$

\subsection{The constructions of  $v_{01}$,~$p_{01}$,~$\theta_{01}$,~$f_{01}$,~$\mathring{R}_{01}$ }

\indent

The constructions in this section are same as in subsection 4.4.
First, we denote $M_1$ by
   \begin{align}
   M_1=&\sum\limits_{l\in Z^3}b^2_{1l}(x,t)\Big{(}B_{k_1}\otimes B_{k_1}e^{2i\lambda_1 2^{|l|} k_1\cdot (x-\frac{l}{\mu_1}t)}+B_{-k_1}\otimes B_{-k_1}
   e^{-2i\lambda_1 2^{|l|} k_1\cdot (x-\frac{l}{\mu_1}t)}\Big{)}\nonumber\\
   &+\sum\limits_{l,l'\in Z^3 ,l\neq l'}w_{1ol}w_{1ol'}(x,t),
   \end{align}


 and
$N_1,K_1$ by
 \begin{align}
   N_1=&\sum_{l\in Z^3}\Big[w_{1l}\otimes \Big(v_0-\frac{l}{\mu_{1}}\Big)
   +\Big(v_0-\frac{l}{\mu_{1}}\Big{)}\otimes w_{1l}\Big],\nonumber\\
   K_1=&\sum\limits_{l\in Z^3}\beta_{1l}(x,t)b_{1l}(x,t)\Big(B_{k_1}e^{2i\lambda_1 2^{|l|} k_1\cdot (x-\frac{l}{\mu_1}t)}+ B_{-k_1}e^{-2i\lambda_1 2^{|l|} k_1\cdot (x-\frac{l}{\mu_1}t)}\Big{)}+\sum\limits_{l,l'\in Z^3 ,l\neq l'}w_{1ol}\chi_{1l'}.
    \end{align}
 Notice that $N_{1}$ is a symmetric matrix.
   Then we define
  \[ \begin{aligned}
    &v_{01}:=v_0+w_1,\\
    &p_{01}:=p_0-\frac{2w_{1o}\cdot w_{1oc}+|w_{1oc}|^2}{3}-tr(N_1),\\
    &\theta_{01}:=\theta_0+\chi_1-\fint_{T^{3}}{\chi_1}dx,\\
     &R_{01}:=-R_0(x,t)+2\sum\limits_{l\in Z^3}b^2_{1l}(x,t)Re(B_{k_1}\otimes B_{-{k_1}})+\delta \mathring{R}_{01},\\
    &f_{01}:=f_0+2\sum\limits_{l\in Z^3}\beta_{1l}(x,t)b_{1l}(x,t)A_{k_1}+\delta f_{01},
   \end{aligned}\]
   where
   \begin{align}
   \delta \mathring{R}_{01}=&\mathcal{R}(\hbox{div}M_1)+N_1-tr(N_1)Id+\mathcal{R}\Big\{\partial_tw_{1}
   +\hbox{div}\Big[\sum_{l\in Z^3}\Big(w_{1l}\otimes\frac{l}{\mu_{1}}+\frac{l}{\mu_{1}}\otimes w_{1l}\Big)\Big]\Big\}\nonumber\\
   &+(w_{1o}\otimes w_{1oc}+w_{1oc}\otimes w_{1o}+w_{1oc}\otimes w_{1oc})
   -\frac{2w_{1o}\cdot w_{1oc}+|w_{1oc}|^2}{3}Id\nonumber\\
   &-\mathcal{R}\Big(\Big(\chi_1-\fint_{T^{3}}{\chi_1}dx\Big)e_3\Big),
   \end{align}

   and
   \begin{align}
   \delta f_{01}=&\mathcal{G}(\hbox{div}K_1)
   +\mathcal{G}(w_1\cdot\nabla\theta_{0})
   +\mathcal{G}\Big(\partial_t\chi_{1}+\sum_{l\in Z^3}\Big(\frac{l}{\mu_{1}}\cdot\nabla\Big)\chi_{1l}\Big)\nonumber\\
   &+\sum_{l\in Z^3}\Big(v_0-\frac{l}{\mu_1}\Big)\chi_{1l}+w_{1oc}\chi_1.
  \end{align}

   Obviously,
   \begin{align*}
    \hbox{div}v_{1}=\hbox{div}v_0+\hbox{div}w_{1}=0.
   \end{align*}
   As in subsection 4.4, we have
\[\begin{aligned}
    \hbox{div}R_{01}=
    &\partial_tv_{01}+\hbox{div}(v_{01}\otimes v_{01})+\nabla p_{01}-\theta_{01}e_3,
    \end{aligned}\]
 and
 \begin{align}
 \hbox{div}f_{01}=
 &\partial_t\theta_{01}+\hbox{div}(v_{01}\theta_{01}).
 \end{align}
 Thus the functions $v_{01},p_{01},\theta_{01},R_{01},f_{01}$ solve the system (\ref{d:boussinesq reynold}).


\indent


\subsection{The representation of $-R_0(x,t)+2\sum\limits_{l\in Z^3}b^2_{1l}(x,t)Re(B_{k_1}\otimes B_{-{k_1}})$}

\indent

This subsection is similar to that of subsection 5.1. From the definition of $b_{1l}(x,t)$ , we have
\begin{align}
2\sum\limits_{l\in Z^3}b^2_{1l}(x,t)Re(B_{k_1}\otimes B_{-{k_1}})
=&\delta\sum\limits_{l\in Z^3}\alpha_l^2(\mu_1 v_0)\gamma_{k_1}^2\Big{(}\frac{R_0(x,t)}{\delta}\Big{)}\Big(Id-
\frac{k_1}{|k_1|}\otimes\frac{k_1}{|k_1|}\Big)\nonumber\\
=&\delta\gamma_{k_1}^2\Big{(}\frac{R_0(x,t)}{\delta}\Big{)}\Big(Id-
\frac{k_1}{|k_1|}\otimes\frac{k_1}{|k_1|}\Big).\nonumber
\end{align}
Where we used $\sum\limits_{l\in Z^3}\alpha_l^2=1$.
Moreover, from the Lemma \ref{p:split} and estimate (\ref{b:near id}), we can decompose $R_{0}$
\begin{align}
\frac{R_0}{\delta}=\sum\limits_{i=1}^{N}\Big(\gamma_{k_i}\Big(\frac{R_0}{\delta}\Big)\Big)^2\Big(Id-
\frac{k_i}{|k_i|}\otimes\frac{k_i}{|k_i|}\Big).\nonumber
\end{align}
Therefore
\begin{align}
&-R_0(x,t)+2\sum\limits_{l\in Z^3}b^2_l(x,t)Re(B_{k_1}\otimes B_{-{k_1}})=-\delta\sum\limits_{i=2}^{N}\Big(\gamma_{k_i}\Big(\frac{R_0}{\delta}\Big)\Big)^2\Big(Id-
\frac{k_i}{|k_i|}\otimes\frac{k_i}{|k_i|}\Big).\nonumber
\end{align}
Meanwhile, we have
\begin{align}
R_{01}=-\delta\sum\limits_{i=2}^{N}\Big(\gamma_{k_i}\Big(\frac{R_0}{\delta}\Big)\Big)^2\Big(Id-
\frac{k_i}{|k_i|}\otimes\frac{k_i}{|k_i|}\Big)+\delta \mathring{R}_{01}.\nonumber
\end{align}

\subsection{The representation of $f_0+2\sum\limits_{l\in Z^3}\beta_{1l}(x,t)b_{1l}(x,t)A_{k_1}$}

\indent

This subsection is similar to that of subsection 5.2, we omit the detailed proof, only give results.
We have
\begin{align}
&f_0+2\sum\limits_{l\in Z^3}\beta_{1l}(x,t)b_{1l}(x,t)A_{k_1}=\sum\limits_{i=2}^{3}g_{k_i}(f_0(x,t))A_{k_i}.\nonumber
\end{align}
Meanwhile, we have
\begin{align}
f_{01}=f_0+2\sum\limits_{l\in Z^3}\alpha_l(\mu_1v_0)\beta_1(x,t)b_l(x,t)A_{k_1}+\delta f_{01}
=\sum\limits_{i=2}^{3}g_{k_i}(f_0(x,t))A_{k_i}+\delta f_{01}.\nonumber
\end{align}

 \subsection{Estimates on $\delta \mathring{R}_{01}$ and $\delta f_{01}$}

 \indent



We summarize the main properties of $b_{1l}$ and $\beta_{1l}$.
\begin{Lemma}
For any $|l|\leq C_1\mu_1$ and $r\geq0$
\begin{align}
\|b_{1l}\|_r\leq C_1\mu_1^r,\\
\|\beta_{1l}\|_r\leq C_1\mu_1^r,\\
\|(\partial_t+\frac{l}{\mu_1}\cdot\nabla)b_{1l}\|_r\leq C_1\mu_1^{r+1},\\
\|(\partial_t+\frac{l}{\mu_1}\cdot\nabla)b_{1l}\|_r\leq C_1\mu_1^{r+1}.
\end{align}
\begin{proof}
The proof is same as in Lemma \ref{l:estimate amp}, we omit it here.
\end{proof}
\end{Lemma}

 \begin{Lemma}[Estimates on corrections]
 For any $\alpha\in [0,1)$,
 \begin{align}
 \|w_{1oc}\|_\alpha\leq C_1\frac{\mu_1}{\lambda_1^{1-\alpha}},\qquad
 \end{align}
 \end{Lemma}
 \begin{proof}
 The proof is same as in Lemma \ref{e: estimate correction}. We omit it here.
 \end{proof}

 \subsubsection{Estimates on $\delta \mathring{R}_{01}$}

 \indent

 We may rewrite
  \begin{align}
   \delta \mathring{R}_{01}=&\mathcal{R}(\hbox{div}M_1)+N_1-tr(N_1)Id+\mathcal{R}\Big\{\partial_tw_{1}+\hbox{div}\Big[\sum_{l\in Z^3}\Big(w_{1l}\otimes\frac{l}{\mu_{1}}+\frac{l}{\mu_{1}}\otimes w_{1l}\Big)\Big]\Big\}\nonumber\\
   &+(w_{1o}\otimes w_{1oc}+w_{1oc}\otimes w_{1o}+w_{1oc}\otimes w_{1oc})\nonumber\\
   &-\frac{2w_{1o}\cdot w_{1oc}+|w_{1oc}|^2}{3}Id
   -\mathcal{R}\Big(\Big(\chi_1-\fint_{T^{3}}{\chi_1}dx\Big)e_3\Big)\nonumber
   \end{align}
  Again, we split the stress into three parts,  \\
  (1) the oscillation part $$\mathcal{R}(\hbox{div}M_1)-\mathcal{R}\Big(\Big(\chi_1-\fint_{T^{3}}{\chi_1}dx\Big)e_3\Big),$$
  (2) the transportation part
  \begin{align}
  &\mathcal{R}\Big\{\partial_tw_{1}
   +\hbox{div}\Big[\sum_{l\in Z^3}\Big(w_{1l}\otimes\frac{l}{\mu_{1}}+\frac{l}{\mu_{1}}\otimes w_{1l}\Big)\Big]\Big\}=\mathcal{R}\Big(\partial_tw_{1}+\sum_{l\in Z^3}\Big(\frac{l}{\mu_{1}}\cdot\nabla\Big)w_{1l}\Big),\nonumber
   \end{align}
  (3) the error part
  \begin{align}
  &N_1-tr(N_1)Id+(w_{1o}\otimes w_{1oc}+w_{1oc}\otimes w_{1o}-w_{1oc}\otimes w_{1oc})-\frac{2w_{1o}\cdot w_{1oc}+|w_{1oc}|^2}{3}Id.\nonumber
   \end{align}
Again, we will estimate each terms separately.

\begin{Lemma}[The oscillation part]
\begin{align}
\|\mathcal{R}({\rm div}M_1)\|_\alpha\leq&C_1\frac{\mu_1^3}{\lambda_1^{1-\alpha}},\qquad \forall\alpha\in (0,1)\nonumber\\
\Big\|\mathcal{R}\Big(\Big(\chi_1-\fint_{T^{3}}{\chi_1}dx\Big)e_3\Big)\Big\|_\alpha\leq& C_1\frac{\mu_1}{\lambda_1^{1-\alpha}}.\nonumber
\end{align}

\end{Lemma}

\begin{Lemma}[The transportation part]
\begin{align}
\Big\|\mathcal{R}\Big(\partial_tw_{1}+\sum_{l\in Z^3}\frac{l}{\mu_{1}}\cdot\nabla w_{1l}\Big)\Big\|_\alpha\leq& C_1\frac{\mu_1^2}{\lambda_1^{1-\alpha}}.\nonumber
\end{align}


\end{Lemma}

\begin{Lemma}[Estimate on error part I]
\begin{align}
\|N_1\|_0\leq \frac{C_1}{\mu_1}.\nonumber
\end{align}
\end{Lemma}

\begin{Lemma}[Estimates on error part II]
\begin{align}
\|w_{1o}\otimes w_{1oc}+w_{1oc}\otimes w_{1o}+w_{1oc}\otimes w_{1oc}\|_0\leq C_1\frac{\mu_1}{\lambda_1}.\nonumber
\end{align}

\end{Lemma}


 The proof of the above four lemmas are same as in Lemma \ref{e:oscillation estimate}, Lemma \ref{e:transport estimate}, Lemma \ref{e:error 1} and  Lemma \ref{e: error 2} respectively. We omit it here.

Finally, we conclude that
\begin{align}
\|\delta\mathring{R_{01}}\|_0\leq C_1\Big(\frac{\mu_1^3}{\lambda_1^{1-\alpha}}+\frac{1}{\mu_1}\Big),\qquad \forall\alpha\in(0,1).
\end{align}
Furthermore,
\begin{align}
\|v_{01}-v_0\|_0\leq& \frac{M\sqrt{\delta}}{2L}+C_1\frac{\mu_1}{\lambda_1},\nonumber\\
\|p_{01}-p_0\|_0\leq& C_1\Big(\frac{1}{\mu_1}+\frac{\mu_1}{\lambda_1} \Big).\nonumber
\end{align}

\subsubsection{Estimates on $\delta f_{01}$}

 \indent

Recall that
\begin{align}
   \delta f_{01}=&\mathcal{G}(\hbox{div}K_1)
   +\mathcal{G}(w_1\cdot\nabla\theta_{0})
   +\mathcal{G}\Big(\partial_t\chi_{1}+\sum_{l\in Z^3}\Big(\frac{l}{\mu_{1}}\cdot\nabla\Big)\chi_{1l}\Big)+\sum_{l\in Z^3}\Big(v_0-\frac{l}{\mu_1}\Big)\chi_{1l}+w_{1oc}\chi_1.\nonumber
\end{align}

As before, we split $\delta f_{01}$ into three parts:\\
(1) the oscillation part
\begin{align}
\mathcal{G}(\hbox{div}K_1)+\mathcal{G}(w_1\cdot\nabla\theta_0).\nonumber
\end{align}
(2) the transport part
\begin{align}
\mathcal{G}\Big(\partial_t\chi_{1}+\sum_{l\in Z^3}\Big(\frac{l}{\mu_{1}}\cdot\nabla\Big)\chi_{1l}\Big).\nonumber
\end{align}
(3) the error part
\begin{align}
w_{10c}\chi_1+\sum\limits_{l\in Z^3}\Big(v_0-\frac{l}{\mu}\Big)\chi_{1l}.\nonumber
\end{align}
We again estimate each of them. And the proof of the following lemmas are same in as Lemma \ref{e:osci},  Lemma \ref{e:trans} and  Lemma \ref{e:est 1} respectively. We omit it here.
\begin{Lemma}[The Oscillation Part]
\begin{align}
\|\mathcal{G}({\rm div}K_1)+\mathcal{G}(w_1\cdot\nabla\theta_0)\|_\alpha\leq C_1\frac{\mu_1^3}{\lambda_1^{1-\alpha}}.
\end{align}

\end{Lemma}

\begin{Lemma}[The transportation Part]
\begin{align}
\Big\|\mathcal{G}\Big(\partial_t\chi_{1}+\sum_{l\in Z^3}\Big(\frac{l}{\mu_{1}}\cdot\nabla\Big)\chi_{1l}\Big)\Big\|_\alpha\leq C_1\frac{\mu_1^2}{\lambda_1^{1-\alpha}}.
\end{align}


\end{Lemma}

\begin{Lemma}[The error Part]
\begin{align}
\Big\|w_{1oc}\chi_1+\sum_{l\in Z^3}\Big(v_0-\frac{l}{\mu_1}\Big)\chi_{1l}\Big\|_0\leq C_1\Big(\frac{\mu_1^2}{\lambda_1}+\frac{1}{\mu_1}\Big).
\end{align}

\end{Lemma}
Finally, we conclude that
\begin{align}
\|\delta f_{01}\|_0\leq C_1\Big(\frac{\mu_1^3}{\lambda_1^{1-\alpha}}+\frac{1}{\mu_1}\Big).\nonumber
\end{align}

\begin{align}
 \|\theta_{01}-\theta_{0}\|_0\leq& \frac{\sqrt{M\delta}}{2L}+C_1\frac{\mu_1^2}{\lambda_1}.\nonumber
\end{align}
And
\begin{align}
 &R_{01}:=-\delta\sum\limits_{i=2}^{L}\Big(\gamma_{k_i}\Big(\frac{R_0}{\delta}\Big)\Big)^2\Big(Id-\frac{k_i}{|k_i|}\otimes\frac{k_i}{|k_i|}\Big)+\delta \mathring{R}_{01},\nonumber\\
    &f_{01}:=\sum\limits_{i=2}^{3}g_{k_i}(f_0(x,t))A_{k_i}+\delta f_{01},\nonumber
\end{align}
with the estimate
\begin{align}
 \|v_{01}-v_{0}\|_0\leq& \frac{M\sqrt{\delta}}{2L}+C_1\frac{\mu_1}{\lambda_1} ,\nonumber\\
\|p_{01}-p_{0}\|_0\leq& C_1\Big(\frac{1}{\mu_1}+\frac{\mu_1}{\lambda_1} \Big),\nonumber\\
 \|\theta_{01}-\theta_{0}\|_0\leq& \frac{M\sqrt{\delta}}{2L}+\frac{C_1}{\lambda_1} ,\nonumber\\
\|\delta\mathring{R_{01}}\|_0\leq& C_1\Big(\frac{\mu_1^3}{\lambda_1^{1-\alpha}}+\frac{1}{\mu_1}\Big),\nonumber\\
\|\delta f_{01}\|_0\leq& C_1\Big(\frac{\mu_1^3}{\lambda_1^{1-\alpha}}+\frac{1}{\mu_1}\Big).\qquad \forall\alpha\in(0,1).\nonumber
\end{align}
Then we complete the first step.

As the constructions of $\tilde{v},~\tilde{p},~\tilde{\theta},~\tilde{\mathring{R}},~\tilde{f}$ in Proposition \ref{p: iterative 1}, we can also construct $\tilde{v},~\tilde{p},~\tilde{\theta},~\tilde{\mathring{R}},~\tilde{f}$ by inductions which satisfies the Proposition \ref{p:iterative 2}.
We omit the detail here.

\appendix

\section{Stationary phase lemma}\label{s:stationary}

We recall here the following facts about operator $\mathcal{R}$ and $\mathcal{G}$, where the operator $\mathcal{R}$ is defined in \cite{CDL2} .(for a proof we refer to \cite{CDL2})

\begin{proposition}\label{p:stat_phase}

(i) Let $k\in Z ^3\setminus\{0\}$ and $\lambda\geq 1$ be fixed.
For any $c\in C^{\infty}(T^3)$ and $m\in N$ we have
\begin{equation}\label{e:average}
\left|\int_{T^3}c(x)e^{i\lambda k\cdot x}\,dx\right|\leq \frac{[c]_m}{\lambda^m}.
\end{equation}

(ii) Let $\varphi_\lambda\in C^{\infty}(T^3)$ be the solution of
$$\triangle\varphi_\lambda=f_\lambda$$
in $T^3$ with $\int_{T^3}\varphi_\lambda dx=0$, where
$$f_\lambda(x):=c(x)e^{i\lambda k\cdot x}-\fint_{T^3}c(y)e^{i\lambda k\cdot y}dy.$$
Then for any $\alpha\in (0,1)$ and $m\in N$ we have the estimate
\begin{align}
\|\nabla\phi_\lambda\|_\alpha\leq C\Big(\frac{\|c\|_0}{\lambda^{1-\alpha}}+\frac{[c]_m}{\lambda^{m-\alpha}}+\frac{[c]_{m+\alpha}}{\lambda^{m}}\Big).\nonumber
\end{align}
where $C=C(\alpha,m)$.

(iii) Let $k\in Z^3\setminus\{0\}$ be fixed. For a smooth vector field $c\in C^{\infty}(T^3;R^3)$ and a smooth scalar function $d\in C^{\infty}(T^3;R)$. Let
$F(x):=c(x)e^{i\lambda k\cdot x}$ and $H(x):=d(x)e^{i\lambda k\cdot x}$. Then we have
\begin{align}\label{p:R G}
\|\mathcal{R}(F)\|_{\alpha}\leq \frac{C}{\lambda^{1-\alpha}}\|c\|_0+\frac{C}{\lambda^{m-\alpha}}[c]_m+\frac{C}{\lambda^m}[c]_{m+\alpha},\\
\|\mathcal{G}(H)\|_{\alpha}\leq \frac{C}{\lambda^{1-\alpha}}\|d\|_0+\frac{C}{\lambda^{m-\alpha}}[d]_m+\frac{C}{\lambda^m}[d]_{m+\alpha},\label{e:G}
\end{align}
where $C=C(\alpha,m)$.
\end{proposition}
\begin{proof}
(i),(ii) and the estimate of $\mathcal{R}$ in (iii) can be found in \cite{CDL2}.
From the definition of $\mathcal{G}$ and (ii), we directly arrive at the estimate (\ref{e:G}).
\end{proof}


{\bf Acknowledgments.}
The research is partially supported by the Chinese NSF under grant 11471320.
We thank Tianwen Luo for very valuable discussions.


\begin{thebibliography}{WWW}

\bibitem{TBU}
T. Buckmaster, {\it Onsager's conjecture almost everywhere in time},
Commun. Math. Phys. 333(2015), 1175-1198.

\bibitem{BCDLI}
T. Buckmaster, C.  De Lellis, P.  Isett,  Szekelyhidi. Jr.
L, {\it Anomalous dissipation for 1/5-Holder Euler
flows},  Ann. of. Math. 182(2015), 127-172

\bibitem{BCDL1}
T. Buckmaster, C.  De Lellis,  Szekelyhidi. Jr.
L, {\it Transporting microsructure and dissipative Euler flows},
arXiv:1302.2825, (2013)

\bibitem{BCDL2}
T, Buckmaster, C.  De Lellis,  Szekelyhidi. Jr.
L, {\it Dissipative Euler
flows with Onsager-critical spatial
regularity},  Common. Pure. Appl .Math. (2015), 1-58

\bibitem{CPFR}
A. Cheskidov, P. Constantin, S. Friedlander, R. Shvydkoy, {\it Energy conservation and Onsager's conjecture for the Euler
equations},  Nonlinearity  21(6)(2008), 1233-1252

\bibitem{CHO}
A. Choffrut, {\it H-principles for the incompressible Euler
equations},  Arch. Rational. Mech. Anal. 210(2013), 133-163.

\bibitem{PC}
P. Constantin, {\it On the Euler equation of incompressible flow},
Bull. Amer. Math. Soc. 44(4)(2007), 603-621.



\bibitem{CM}
P. Constantin, A. Majda, {\it The Betrami spectrum for incompressible fluid flows},
 Commun. Math. Phys. 115(1988), 435-456

\bibitem{CCDE}
S. Conti, C. De Lellis, Jr. L. Szekelyhidi,  {\it H-principle and
rigidity for $C^{1,\alpha}$ isometric embeddings} , In Nonlinear Partial Differential
 Equations vol.7 of Abel Symposia Springer (2012), 83-116.


\bibitem{CET}
P. Constantin, E. W, Titi. E. S, {\it Onsager's conjecture on
the energy conservation for solutions of Euler's equation}, Comm. Math. Phys,
 165(1)(1994), 207-209.



\bibitem{DA}
S. Daneri, {\it Cauchy problem for dissipative Holder solutions to the
incompressible Euler equations},  Commun. Math. Phy. (2014), 1-42.


\bibitem{CDL}
C. De Lellis, Jr. L. Szekelyhidi, {\it The Euler equation as a differential inclusion}, Ann. of. Math. 170(3)(2009), 1417-1436.

\bibitem{CDL0}
C. De Lellis, Jr. L. Szekelyhidi, {\it On admissibility criteria
for weak solutions of the Euler equations},  Arch. Ration. Mech. Anal. 195(1)(2010), 225-260.


\bibitem{CDL1}
C. De Lellis, Jr. L. Szekelyhidi,   {\it The h-principle and the
equations of fluid dynamics}, Bull. Amer. Math. Soc. 49(3)(2012), 347-375.


\bibitem{CDL2}
C. De  Lellis, Jr. L. Szekelyhidi, {\it Dissipative continuous
Euler flows}, Invent. Math. 193(2)(2013), 377-407


\bibitem{CDL3}
C. De Lellis, Jr. L. Szekelyhidi, {\it Dissipative Euler flows and
Onsager's conjecture}, Jour. Eur. Math. Soc.(JEMS)16(2014),no. 7, 1467-1505.

\bibitem{DUR}
J. Duchon, R. Raoul, {\it Inertial energy dissipation for weak solutions of incompressible Euler and Navier-Stokes equations},
Nonlinearity. 13(2000), 249-255



\bibitem{ISOH1}
P. Isett, Oh, S.-J, {\it A heat flow approach to Onsager's conjecture for the Euler equations on manifolds},
to appear in Trans. Amer. Math. Soc

\bibitem{ISOH2}
P. Isett, Oh, S.-J, {\it On nonperiodic Euler flows with H\"{o}lder regularity}, arxiv:1402.2305v2, 2015


\bibitem{IS1}
P. Isett,  {\it Holder continuous Euler
flows in three dimensions with
compact support in time}, arXiv:1211.4065, 2012.

\bibitem{ISV2}
P. Isett, V. Vicol, {\it Holder conyinuous solutions of active scalar equations},
arXiv:1405.7656,2014

\bibitem{Ma}
A. Majda,  {\it  Introduction to PDEs and Waves for the Atmosphere and Ocean},
Courant
Lecture Notes in Mathematics, Vol. 9. AMS/CIMS, 2003

\bibitem{NASH}
J. Nash, {\it $C^{1}$ isometric embeddings},  Ann. of. Math. 60(1954), 383-396.

\bibitem{Pe}
J. Pedlosky, {\it Geophysical fluid dynamics}, Springer, New-York, 1987

\bibitem{VS}
V. Scheffer,  {\it An inviscid
ow with compact support in space-time},
J. Geom. Anal. (1993),343-401.


\bibitem{ASH1}
A. Shnirelman, {\it Weak solution with decreasing energy of incompressible Euler equations}, Commun. Math. Phys. 210(2000), 541-603

\bibitem{ASH2}
A. Shnirelman, {\it On the nonuniqueness of weak solution of Euler equation}, Comm. Pure. Appl. Math. 50(12)(1997), 1261-1286





















\bibitem{SH}
R. Shvydkoy, {\it Lectures on the onsager conjecture},  Dis. Con. Dyn. Sys. 3(3)(2010), 473-496.












\end{thebibliography}
\end{document}